\documentclass[12pt,twoside]{amsart}
\usepackage{amsmath,amssymb,mathrsfs,graphicx,bbm,amsfonts,amsthm, bm}
\usepackage [latin1]{inputenc}
\usepackage{color}
\usepackage{todonotes}

\usepackage{comment}

\usepackage{graphicx}
\usepackage{dsfont}

\DeclareMathAlphabet{\mathbfit}{OML}{cmm}{b}{it}

\makeatletter
\@addtoreset{equation}{section}
\makeatother

\theoremstyle{thmstyleone}%
\newtheorem{theorem}{Theorem}
\newtheorem{proposition}[theorem]{Proposition}%

\theoremstyle{thmstyletwo}%

\newtheorem{example}{Example}%
\newtheorem{remark}{Remark}%
\newtheorem{lemma}{Lemma}%
\newtheorem{corollary}{Corollary}%
\newtheorem*{claim*}{Claim}
\newtheorem{assumption}{Assumption}

\theoremstyle{thmstylethree}%
\newtheorem{definition}{Definition}%
\newtheorem*{definition*}{Definition}

 \def \R {{\mathbb R}}
 
 \def \N {{\mathbb N}}
 \def \P {{\mathbb P}}
 
 \def \E {{\mathbb E}}

 \def \cE {\mathcal{E}}

 \def \cL {\mathcal{L}}
 \def \cM {\mathcal{M}}

 \def \cR {\mathcal{R}}
 
 \def \cT {\mathcal{T}}
 
 \def \cV {\mathcal{V}}

 \def \z {{\z}}

 \def \z {{\zeta}}

\def \à {{\`{a}}}
\def \ì{{\`{\i}}}
\def \ò{{\`{o}}}
\def \è{{\`{e}}}
\def \ù{{\`{u}}}

\usepackage{mathtools}
\usepackage{enumitem}

\usepackage{physics}
\usepackage{hyperref}

\def\df{\,\mathrm{d}}
\def\exp{\mathrm{e}^}

 \def \brange{{(0,2)\hspace{-0.1cm}\setminus\hspace{-0.1cm} \{1\}}}

\begin{document}
\title[Precise large deviations through a uniform Tauberian theorem]{Precise large deviations through\\ a uniform Tauberian theorem}

\author{Giampaolo Cristadoro}
\address{Dipartimento di Matematica e Applicazioni,
Universit\`a di Milano-Bicocca,
Via R. Cozzi 55, 20125 Milano, Italy.}
\email{giampaolo.cristadoro@unimib.it}
\author{Gaia Pozzoli}
\address{Department of Mathematics, CY Cergy Paris University, CNRS UMR 8088, 2 avenue Adolphe Chauvin, 95302 Cergy-Pontoise, France. }
\email{gaia.pozzoli@cyu.fr}

\date{}

\begin{abstract}
We derive a large deviation principle for families of random variables in the basin of attraction of spectrally positive stable distributions by proving a uniform version of the Tauberian theorem for Laplace--Stieltjes transforms. The main advantage of this method is that it can be easily applied to cases that are beyond the reach of the techniques currently used in the literature. Notable examples include large deviations for random walks with long-ranged memory kernels, as well as for randomly stopped sums where the random time $\mathrm{N}$ is either not concentrated around its expectation or has an infinite mean. The method reveals the role of the characteristic function when Cram\'er's condition is violated and provides a unified approach within regular variation.

    \end{abstract}

\maketitle

\section{Introduction}\label{Intro}
The theory of large deviations \cite{Dembo, DenHollander, Ellis, Petrov} investigates the decay rate of probabilities of ``rare'' events. It provides a quantitative description of the probability of events that typically, but not necessarily, concern sums of random variables significantly deviating from expectations. 
To make precise the terms ``fluctuations'' and ``rare events''  in our field of interest, consider a sequence of real-valued random variables $(X_k)_{k=1}^\infty$ and define the corresponding cumulative sum process $S_n\coloneqq \sum_{k=1}^n X_k$. Let $\{a_n\}_{n\in\N}$ and $\{b_n\}_{n\in \N}$ be two sequences with $a_n>0$, $b_n\in \R$ and assume the following convergence in probability 
\begin{equation}\label{eq:rareEvent}
\frac{S_n-b_n}{a_n}\xrightarrow{ \:p\: }0.
\end{equation}
Then $\P[|S_n-b_n|>a_n]$, $\P[S_n-b_n>a_n]$ and $\P[S_n-b_n<-a_n]$ are called \emph{probabilities of large deviations} of the partial sums $S_n$.

By way of illustration, these large deviation probabilities naturally appear in insurance and finance, where the proper functioning of a system is threatened by the occurrence of a rare event \cite{Esscher32}. But plenty of applications can be found across diverse fields such as analysis (with heat conduction in random media; see~\cite{Gartner99}), probability (with random walks in random environments; refer to~\cite{Greven94}), statistical physics and queueing theory (with fluctuations in renewal-reward processes; see for instance~\cite{Zamparo2023}), to cite a few.

\medskip
The paper is organized as follows. In Sections~\ref{subsec:cramer}--\ref{subsec:reg} we properly introduce historical context, notation and objects of interest. In Section~\ref{sec:taub} we present our main result: we state a \emph{uniform} Tauberian theorem for families of bilateral Laplace--Stieltjes transforms (Theorem \ref{thm:tauberian}).   As its direct by-product, we formulate a precise large deviation principle for sequences of regularly varying random variables stemming  from a proper uniform control on the behaviour  near the origin of their --- non-analytic ---   Laplace--Stieltjes transforms (Theorem~\ref{thm:LDP}). 
In Section~\ref{sec:Appl} we show that this uniform control is possible in different models, and we derive the corresponding large deviation results from the application of our Theorem~\ref{thm:LDP}. 
The discussion of the proofs of the main theorems is postponed to Section~\ref{sec:proofs}.

\subsection{Classical large deviations: Cramer's condition}\label{subsec:cramer}
Historically, it is convenient to start with a prime example of traditional importance. Let $(X_k)_{k\in \N}$ be  independent and identically distributed (i.i.d.) random variables having a common distribution function $F$ with expected value $\mu$ and finite variance $\sigma^2$. As a consequence of the \emph{Central Limit Theorem} (CLT), the large deviation probabilities with $b_n\equiv n\mu$ and $a_n\equiv \sigma\sqrt{n}\cdot x_n$ (for some $x_n\to \infty$ as $n \to \infty$) depict how the partial sum differs from its mean by an amount that is well-beyond the ``normal'' fluctuations identified by the natural-scale sequence $\sqrt n$. 
The foundational result appearing in Cram\'er's pioneering work \cite{Cramer1938} fits into this setting under the additional assumption, also known as Cram\'er's condition, that there exists the moment generating function of $F$ in a neighborhood of the origin
\begin{equation}\label{eq:CramerCond}
\cM_F(h)\coloneqq\int_{-\infty}^\infty \exp{hx}\df F(x)<\infty\quad\text{for all}\quad |h|<A,\quad \text{where}\quad A>0.
\end{equation}
This hypothesis is responsible for an exponential decay of the probabilities of large deviations from the law of large numbers, with rate function obtained as the Legendre--Fenchel transform of the cumulant generating function of $F$\footnote{It is evident that for $\sqrt n\ll a_n\ll n$ neither the CLT ($a_n\asymp \sqrt n$) nor Cram\'er's statement ($a_n\asymp n$) embed the asymptotic behaviour of the convergence~\eqref{eq:rareEvent} (where $f(x)\asymp g(x)$ if and only if $\exists\,c>0$ such that $\lim_{x\to\infty}f(x)/g(x)=c$). This is the realm of the so-called \emph{moderate deviations} \cite[\S 3.7]{Dembo}, where a universal quadratic rate function appears: fine features in $\cM_F$ are not relevant in this range of scales as in the CLT regime,
\[
\lim_{n\to\infty}\frac n {a_n^2}\log\P[S_n-\mu n>c a_n]=-\frac{c^2}{2\sigma^2}\quad\text{for all}\quad c>0.
\]
} 
\[
\lim_{n\to\infty}\frac 1 n \log \P[S_n-\mu n>c n]=-\sup_{u\in\R}\{(\mu+c)u-\log \cM_F(u)]\}\quad\text{for all}\quad c>0.
\]
For a more in-depth discussion with the general concept of a large deviation principle and related ideas such as the Cram\'er's series, that is beyond the scope of this article, refer to~\cite{Dembo, Petrov}.

\noindent 
This \emph{rough} (at logarithmic scale) large deviation result has been subsequently generalized also to weak dependent sequences by G\"artner and Ellis \cite{Gartner77,Ellis84}, who further emphasized the key role played by the moment generating function of $S_n$ in the decay of rare event probabilities (see also~\cite{Bingham2008}).

\subsection{Subexponentiality and the big-jump domain}
However, in several contexts ranging from mathematical finance~\cite[Ch.~8]{EKM} and time series analysis~\cite{Kulik} to queueing theory~\cite{Pakes75} and traffic models~\cite[Ch.~8]{Resnick}, distributions with exponentially decaying tails induced by the Cram\'er's condition are not suited for realistic modeling. 
When assumption~\eqref{eq:CramerCond} is violated, the distributions are said to be \emph{heavy-tailed}~\cite{FKZ}, since we can equivalently write, by defining the tail distribution $\bar F(x)\coloneqq \int_x^\infty \df F(x)$,
\[
\limsup_{x\to\infty}\exp{hx}\bar F(x)=\infty\quad\text{for all}\quad h>0.
\]
In the aforementioned classical case, an extremely large value of the sum of i.i.d. light-tailed random variables was essentially determined by the collection of several ``small increments''. Different mechanisms, instead, are behind the appearance of extreme events when $(X_k)_{k\in\N}$ are i.i.d. with heavy-tailed distribution $F$, and will influence large deviations statements as well.
For explanatory purposes, let us introduce the family of the so-called \emph{subexponential} distributions \cite{Chistyakov64,Embrechts80}: given a common heavy-tailed distribution function $F$ on $[0,\infty)$\footnote{For an extension to distributions on the entire real line $\R$ refer to~\cite[Ch.~3]{FKZ}}, the following additional regularity property is required\footnote{Throughout this paper, given two functions $f(x)$ and $g(x)$ we write $f(x)\sim g(x)$ if and only if $\lim_{x\to\infty}f(x)/g(x)=1$.}
\[
\P[S_n>x]\sim n\,\bar F(x)\sim\P[\max\{X_1,\dots,X_n\}>x]\quad\forall\, n\geq 1 \quad\text{as}\quad x\to\infty.
\]
In this special but very broad class, where tail distributions are well-behaved under convolutions, the extreme event for the partial sum is dominated by a single\,---\,the largest\,---\,summand, and this phenomenon is reflected in the use of the expression ``big-jump principle''.
No large deviation results in this context were obtained until the 1960's, when a new interest arose with works by \cite{Nagaev65, Heyde67a, Heyde67b, Heyde68, Nagaev69a, Nagaev69b, Nagaev79, Nagaev82, ClineHsing, Mikosch98,Ng2004, Denisov2008}. 
Although under possibly different hypotheses, all the statements can be summarized by the following \emph{precise} large deviation result: given $(x_n)_{n\in\N}$ such that $x_n\to\infty$ and $ n\bar F(x_n)\to 0$ as $n\to\infty$, 
\begin{equation}\label{eq:PLD}
\lim_{n\to\infty}\sup_{x\geq x_n}\left|\frac{\bar F_n(x)}{n\,\bar F(x)}-1  \right|= 0,
\end{equation}
where $\bar F_n$ denotes the tail distribution of the centered partial sum $S_n-b_n$. The term ``precise'' stresses the presence of an exact asymptotic equivalence, as opposed to  Cram\'er-type estimates at logarithmic scale.

\subsection{Regular variation and Tauberian theorems}\label{subsec:reg}
In the following, we will focus on the regularly varying tail distributions, which are a subclass of  subexponential distributions, with the aim of proposing a new approach to the corresponding theory of large deviations.
Although precise large deviation principles are available in the literature for subexponential families more general than the regularly varying ones \cite{ClineHsing, Ng2004, Denisov2008}, this restriction allows us to highlight a new perspective on the subject.

\medskip
\begin{definition*}
A tail distribution function $\bar F$ is \emph{regularly varying with index $-\alpha$}, $\bar F\in \cR\cV_{-\alpha}$, if 
\[
\lim_{x\to\infty} \frac{\bar F(\Lambda x)}{\bar F(x)}=\Lambda^{-\alpha}\quad\text{for every}\quad \Lambda>0\quad\text{and some}\quad \alpha \geq 0,
\]
or equivalently if $\bar F(x)=x^{-\alpha} \ell(x)$ with $\ell(x)$ slowly varying at infinity, that is $\lim_{x\to\infty}\ell(\Lambda x)/\ell(x)=1$, which means $\ell(x)\in \cR\cV_0$.
\end{definition*}

\medskip
\noindent
Whereas the characteristic function naturally arises in the study of fluctuations for Cram\'er-type situations, to our knowledge no explicit reference can be found in the literature for the framework of regularly varying distributions. Classical proofs are essentially based on direct probabilistic estimates and the use of a standard truncation argument. Our purpose, instead, is to shed light on the role played by characteristic functions that are not analytic at the origin, in order to parallel Cram\'er's technique in the heavy-tailed context; see Theorem~\ref{thm:LDP}.

The classical Tauberian theory \cite{Bingham} concerns the derivation of the asymptotic behaviour of a function from the knowledge of asymptotic estimates on a properly defined integral transform. Its application to the sum at fixed $n$ of i.i.d. random variables with regularly varying tails leads back to the single big-jump principle previously discussed. In the same vein, we will show (see Theorem~\ref{thm:tauberian} and Section~\ref{sec:proofs}) that a fine control on an indexed family of integral transforms translates into additional information about the original sequence of functions, with remarkable implications for the study of large deviations; see Theorem~\ref{thm:LDP} and Section~\ref{sec:Appl}.

\section{Statement of the main results}
\label{sec:taub}
We start by introducing some notation. 
Let $(G_t(x))_{t\geq0}$ be a family of real-valued functions locally of bounded variation on $\R$ 
(refer to~\cite{Bingham})  
and denote by 
\[
\hat G_t(s)\coloneqq \int_{-\infty}^\infty \exp{-sx}\df G_t(x)
\]
the corresponding family of Laplace--Stieltjes transforms, that are assumed to exist in some right neighborhood of the origin $s\in(0,\xi_t)$. Then consider a family $(L_t(x))_{t\geq0}$ of slowly varying functions satisfying:

\medskip
\begin{definition}\label{def:SVIP}
We write $(L_t)_{t\geq0}\in \cL$ if $(L_t(x))_{t\geq0}$ is a \emph{family of slowly varying functions at infinity} with the following properties: there exists $x_0$, $t_0$ such that $L_t(x)>0$ for all $t\geq t_0$ and $x\geq x_0$; for every given $\Lambda>0$ and $\eta>0$, there exists $\bar x=\bar x(\eta,\Lambda)$ and $\bar t=\bar t(\eta,\Lambda)$ with $\bar t\geq t_0$ and $\min\{\bar x, \Lambda \bar x\}\geq x_0$ such that
\[
\sup_{x\geq \bar x }\left|\frac{L_t(\Lambda x)}{L_t(x)}-1 \right|<\eta\qquad \forall \,t\ge \bar t\,.
\] 
Furthermore, we say that $(L_t)_{t\geq0}\in\cL_\alpha^+\subseteq \cL$ if also the following property holds eventually in $x$ independently of $t$: $x^\alpha L_t(x)$ is non-decreasing.
\end{definition}

\medskip
\begin{remark}\label{rmk:SV}
We note explicitly that if $(L_t(x))_{t\geq 0}$  is a family of separable functions of the two independent variables $L_t(x)\coloneqq c_t \ell(x)$, with $(c_t)_{t\geq0}$ a family of eventually positive constants and $\ell(x)$ an eventually positive slowly varying function at infinity, then  $(L_t)_{t\geq0}\in \cL$. Moreover, without loss of generality, we can always assume that $(L_t)_{t\geq0}\in\cL_\alpha^+$ for every $\alpha>0$ by the monotone equivalents theorem~\cite[Thm.~1.5.3]{Bingham}. This is often the case in applications (see Section~\ref{sec:Appl}): we will refer to this class as $(L_t)_{t\geq0}\in\cL^+_\text{sep}$. 
\end{remark}

\medskip
The following result represents an extension of the Karamata's classical Tauberian theorem for a one-sided Laplace--Stieltjes transform~\cite[Thm.~1.7.1]{Bingham} to a family of bilateral Laplace--Stieltjes transforms indexed by a continuous (or countable) parameter. 

\medskip
\begin{theorem}[\emph{uniform} Tauberian theorem] \label{thm:tauberian}
Let $G_t(x)$ be \emph{non-decreasing} $\forall \, t\geq0$ with $\lim_{x\to-\infty}G_t(x)=0$. Suppose that there exists $(s_t)_{t\geq 0}$ with $s_t>0$ for all $t$, $s_t=o(\xi_t)$  and $s_t\to 0$ as $t\to\infty$ such that, $\forall\,\lambda>0$ and some $\alpha\geq 0$, the family $(\hat G_t(s))_{t\geq 0}$ satisfies 
\begin{equation}
\label{eq:HypTaub}
\lim_{t\to\infty}\sup_{s\in(0,\lambda s_t]}\left |\frac{\hat G_t(s)}{\Gamma(\alpha+1)L_t(1/s)s^{-\alpha}}-1\right|=0\,,
\end{equation}
where $(L_t)_{t\geq 0}\in\cL_\alpha^+$.
Then, defining $x_t\coloneqq 1/s_t$,
\[
\lim_{t\to\infty} \sup_{x\geq x_t}\left|\frac{G_t(x)}{L_t(x)x^\alpha}-1 \right|=0.
\]
\end{theorem}

\medskip
\begin{remark}
A formulation of Tauberian results for a family of one-sided Laplace--Stieltjes transforms indexed by some non-trivial index set $\mathcal I$ can be found in~\cite[\S~2]{Lai76}.
In that setting, as noticed by the author, the statement is a straightforward consequence of the case where $\mathcal I$ is a singleton. However, the hypothesis~(2.2) of~\cite[Thm.~3]{Lai76}  fails for the families considered in this article, which require a finer control in order to derive a Tauberian result.
\end{remark}

\medskip
The following theorem, which parallels Cram\'er-type statements, is a consequence of Theorem~\ref{thm:tauberian}, and provides a general large deviation principle valid for several models under regular variation. 
Hereinafter, we will restrict our attention to the families of slowly varying functions considered in Remark~\ref{rmk:SV}.

\medskip
\begin{theorem}\label{thm:LDP}
Let $(Z_t)_{t\geq 0}$ be a family of $\R$-valued random variables with distribution functions $(F_t(x))_{t\geq 0}$, whose Laplace--Stieltjes transforms satisfy, for $\beta \in \brange$,
\[
\hat F_t(s)=1-\Gamma(1-\beta)L_t(1/s)s^\beta+\mathcal E_t(s),\quad s\ge0,
\]
where, for fixed $t$,  $\mathcal E_t(s)=o(L_t(1/s)s^{\beta})$ as $s\to 0^+$ and $(L_t)_{t\geq 0}\in\cL^+_\text{sep}$. 
Assume that $Z_t\geq 0$ if $\beta\in(0,1)$.
Consider  $(s_t)_{t\geq 0}$ with $s_t>0$ for all $t$ and $s_t\to 0$ as $t \to \infty$ such that $\forall\,\lambda>0$
\begin{equation}\label{eq:controlE}
\lim_{t\to\infty}\sup_{s\in(0, \lambda s_t]}\frac{|\mathcal E_t(s)|}{L_t(1/s)s^\beta}=0.
\end{equation}
Then, for $x_t\coloneqq 1/s_t$, the following precise large deviation principle holds:
\[
\lim_{t\to\infty}\sup_{x\geq x_t}\left|\frac{\P[Z_t>x]}{L_t(x)x^{-\beta}}-1  \right|=0.
\]
\end{theorem}

\noindent
Note that~\eqref{eq:controlE} is a natural generalization of the condition of the classical Tauberian theorem, which is valid for a single element of the family: while for fixed $t\geq 0$ we have that $\cE_t(s)=o(s^\beta)$, here the aim is to find a right neighborhood of the origin where $\cE_t(s)$ is sub-leading uniformly for the whole family.

\noindent
Recall that a stable law~\cite[Ch.~2]{Ibragimov} with index $\beta\in(0,2)$ is called \emph{spectrally positive} if its L\'evy measure is supported on $(0,\infty)$. For stable distributions, this condition is necessary and sufficient for the Laplace transform to be finite in a neighbourhood of the origin. The expansion of $\hat F_t(s)$ in Theorem~\ref{thm:LDP} is precisely the one arising for distributions in the basin of attraction of such stable laws, which explains why Laplace--Stieltjes techniques naturally apply in this setting.

\medskip
\begin{remark}
In applications, the conditions ensuring the validity of~\eqref{eq:controlE} will play the role of the so-called \emph{large deviation conditions}.
\end{remark}

\medskip
\begin{remark}
In the classical scenario of sums of i.i.d. random variables, the large deviation probabilities are given by the product of the number of summands and the tail distribution of the single increment (see~\eqref{eq:PLD}). A similar linear behaviour in the number of summands, albeit with a non-trivial correction on the proportionality factor, has been already observed in a model with correlated increments~\cite{Konst2005}. On the other hand, it is apparent from Theorem~\ref{thm:LDP}  that different asymptotic behaviours are not ruled out. As an example, in the next section we will consider a model with correlated summands where the large deviation probabilities have a non-linear scaling in the number of increments; refer to Corollary~\ref{cor:InonID} and the subsequent discussion. 
\end{remark}

\section{Applications}
\label{sec:Appl}
In this  section, we will apply Theorem~\ref{thm:LDP} to investigate large deviations  for well-established and less classical random walk models stemming both from physics and math communities.
By way of illustration, jointly with the random walk theory, regular variation provides a powerful framework to describe phenomena of anomalous diffusion, that are extensively observed in physics, biology and chemical kinetics. Superdiffusive models, indeed, require that clusters of jumps on a small scale are alternated with larger increments: early examples can be found in~\cite{Mandelbrot82, Montroll65, Shlesinger82, Klafter87}.
Equally, regular variation and large deviation theory are essential for several remarkable applied issues in finance and insurance~\cite{EKM}, since their business policies are shaped on the analysis of the risk function when large claims occur.

\medskip
More specifically, we will derive large deviation estimates for two emblematic classes of models.  These examples, besides having interest in their own, make apparent  the usefulness of a unified approach.
In the first part, we obtain the large deviation probabilities for the sum of independent, but not necessarily identically distributed, random variables with regularly varying tail distributions (see Proposition~\ref{prop:RVI}) for a larger class than previously known. 
As a byproduct, our result allows us to consider  non-trivial examples  of random walk with dependent increments (see Corollary~\ref{cor:InonID}). 
In the second part, we apply our method to the case where the sum is stopped at a random time independent of the sum (see Corollaries~\ref{cor:RSFM} and~\ref{cor:RSIM}).  
Our method permits to deal with a different setting of stopping time, enriching the literature on the subject.
\\

\medskip
Throughout this section, $((X_{k,n})_{k=1}^n)_{n\in\N}$ will denote a row-wise triangular array. Given $n\in\N$, we assume that $ X_{1,n}, X_{2,n}, \dots, X_{n,n}$ are independent, non-negative random variables, each belonging to $\cR\cV_{-\beta_{k,n}}$ with index $\beta_{k,n} \in \brange$. Moreover, $S_n=\sum_{k=1}^n X_{k,n} $ will stand for the corresponding $n$-th partial sum.
 
\noindent
We define
\begin{equation}\label{eq:b}
 b_n\coloneqq\begin{cases}
0\qquad&\text{if}\quad\displaystyle\min_{1\leq k\leq n}\beta_{k,n}\in(0,1),\\
\E[S_n]\qquad&\text{if}\quad\displaystyle\min_{1\leq k\leq n}\beta_{k,n}\in(1,2),
\end{cases}
\end{equation}
and we will denote by $\hat{F}^{(k,n)}(s)=1+ a_{k,n}(s)$  the Laplace--Stieltjes transform of the distribution function of either $X_{k,n} $ if $b_n=0$\,---\,in which case $a_{k,n}(s)< 0$ by definition ($a_{k,n}(s)=0$ only if $X_{k,n}=0$ a.s.)\,---\,or $X_{k,n}-\E[X_{k,n}]$ otherwise\,---\,in which case  $a_{k,n}(s)>0$ for sufficiently small $s>0$ ($a_{k,n}(s)=0$ only if $\operatorname{Var}(X_{k,n})=0$). Moreover, we will use the notation $A_n(s)\coloneqq\sum_{k=1}^n a_{k,n}(s)$.

\subsection{Sums with independent increments}\label{sec:NRsums}
First formulations for large deviation probabilities of independent but not identically distributed summands attracted to the normal law can be derived from~\cite{Gartner77,Ellis84}. Subsequently, the result was extended in~\cite{BollLit1} to independent random variables with finite mean but satisfying the generalized central limit theorem (GCLT), under the additional hypothesis that the average of their tail distribution functions $\bar F^{(j)}$ converges to some limit distribution $\bar F_\infty \in\cR\cV_{-\beta} $
\begin{equation}\label{eq:almostIID}
\lim_{n\to\infty} \frac{\frac 1 n \sum_{j=1}^n \bar F^{(j)}(x)}{\bar F_\infty(x)}=1,\qquad \left( \text{and}\quad\lim_{n\to\infty} \frac{\frac 1 n \sum_{j=1}^n  F^{(j)}(-x)}{ F_\infty(-x)}=1\right)
\end{equation}
uniformly in $x\geq \tilde x$ for some $\tilde x>0$.
This assumption essentially makes precise the concept of an approximately i.i.d. situation. 
Here we provide an extension of these results  including the range $\beta\in(0,1)$, broadening the class of not identical distribution functions stemming from~\eqref{eq:almostIID} and reaching the optimal boundary of the big jump domain.
\medskip

We will use the following generalisation of the almost i.i.d. condition~\eqref{eq:almostIID}:

\medskip
\begin{assumption}\label{assumptionRVI2}  
There exist $\beta\in\brange$, $(L_n)_{n\in\N}\in \cL^+_\text{sep}$ and a sequence $s_n\searrow0$ such that, $\forall\, \lambda>0$ 
$$
\sup_{s\in(0, \lambda s_n]}  \frac{A_n(s)}{-\Gamma(1-\beta)L_n(1/s)s^{\beta}}  \longrightarrow 1 \quad \textrm{as} \quad n\to \infty.
$$
\end{assumption}

\noindent
Note that Assumption~\ref{assumptionRVI2} implies the following two facts:
\begin{enumerate}[label=(\roman*)]
\item $\beta\equiv\min_{1\leq k\leq n}\beta_{k,n}$;
\item For all $n$ large enough, there exists $\mathcal I_n\subseteq\{1,2,\dots,n\}$ such that $\beta_{k,n}\equiv \beta$ for all $k\in\mathcal I_n$ and $\# \mathcal I_n\asymp_1 n$.\footnote{We write $f(x) \asymp_1 x$ if $\exists c\in(0,1]$ such that $\lim_{x\to\infty}f(x)/x=c$.} Hence, for $k\notin\mathcal I_n$, the hypothesis on $X_{k,n}$ can be relaxed: any distribution function decaying faster than $\cR\cV_{-\beta}$ is allowed.
\end{enumerate}

\medskip
\begin{proposition}\label{prop:RVI}
Let $((X_{k,n})_{k=1}^n)_{n\in\N}$ be a triangular array of nonnegative, row-wise independent random variables satisfying Assumption~\ref{assumptionRVI2} with $s_n$ such that $A_n(s_n)\to0$ as $n\to\infty$. 
Then, for $x_n=1/s_n$ we have
\[
\lim_{n\to\infty}\sup_{x\geq x_n}\left|\frac{\P[S_n-b_n>x]}{L_n(x)\,x^{-\beta}}-1  \right|= 0.
\]
\end{proposition}
\begin{proof} We will apply Theorem~\ref{thm:LDP} to the sequence $Z_n\coloneqq S_n-b_n$ with $\hat{F}_n(s)=\prod_{k=1}^n \left(1+a_{k,n}(s)\right)$. 
Observe that, according to the convention in~\eqref{eq:b}, all $a_{k,n}(s)$ have the same sign.
Thus, using Assumption~\ref{assumptionRVI2}, for all $\lambda>0$, for any $(\tilde s_n)_{n\in \N}$ such that $\tilde s_n\leq \lambda s_n$ we have $|a_{k,n}(\tilde s_n)|<1$  eventually in $n$.  
We can apply  Bernoulli's inequality so that, for all $n$ large enough,  we have
\begin{equation}\label{eq:prop3}
0\le  \hat{F}_n(\tilde s_n) -1 - A_n(\tilde s_n)\le \exp{A_n(\tilde s_n)}-1 -A_n(\tilde s_n).
\end{equation}
\color{black}
From~\eqref{eq:prop3} and Assumption~\ref{assumptionRVI2},  we directly deduce (see Lemma~\ref{lem:eqPointUnif}) that, $\forall\, \lambda >0$, 
\begin{align*}
\lim_{n\to\infty}&\sup_{s\in(0,\lambda s_n]}\frac{|\mathcal E_n(s)|}{|\Gamma(1-\beta)|L_n(1/s)s^{\beta}} 
= \lim_{n\to\infty} \sup_{s\in(0,\lambda s_n]}\left|\frac{\mathcal E_n(s)}{A_n(s)}\right|\\
& =\lim_{n\to\infty}\sup_{s\in(0,\lambda s_n]}\left|\frac{\hat F_n(s)-1+ \Gamma(1-\beta)L_n(1/s)s^{\beta} }{ A_n(s)}\right|  \\
&\leq\lim_{n\to\infty}\sup_{s\in(0,\lambda s_n]}\left(\left|\frac{\hat F_n(s)-1 }{A_n(s)}-1\right|  + \left| 1+\frac{\Gamma(1-\beta)L_n(1/s)s^{\beta}}{A_n(s)}\right| \right)=0.
\end{align*}
\end{proof}

\begin{remark}
Proposition~\ref{prop:RVI} remains valid for $\R$-valued, row-wise independent random variables with a finite bilateral Laplace--Stieltjes transform when $\beta\in(1,2)$; see Theorem~\ref{thm:LDP}. No further mention of this observation will be made in the upcoming applications.
\end{remark}

\medskip\noindent
As a first corollary, we immediately derive the well-established result for the i.i.d. subcase (see~\cite[Thm.~3.3]{ClineHsing}):
\medskip
\begin{corollary}\label{cor:IID} 
Let $S_n\coloneqq \sum_{k=1}^nX_k$ with $(X_k)_{k\in\N}$ an i.i.d. sequence of $\R^+$-valued random variables in the class $\cR\cV_{-\beta}$ with $\beta \in \brange$.
Then, for $(x_n)_{n\in\N}$ such that $x_n\to\infty$ and $ n\,\P[X_1>x_n]\to 0$ as $n\to\infty$ we have
\[
\lim_{n\to\infty}\sup_{x\geq x_n}\left|\frac{\P[S_n-b_n>x]}{n\,\P[X_1>x]}-1  \right|= 0.
\]
\end{corollary}
\medskip
\begin{proof} By assumption, $\P(X_1>x) \sim \ell(x)x^{-\beta}$ for some positive slowly varying function $\ell(x)$, and we have a common  $a(s)=-\Gamma(1-\beta)\ell(1/s)s^\beta+o(\ell(1/s)s^\beta)$ as $s\searrow 0$; refer to~\cite[Thm.~2.6.1 and Thm.~2.6.5]{Ibragimov}. The hypotheses of Proposition~\ref{prop:RVI} are  thus readily satisfied with $L_n(1/s)=n\ell(1/s)$ and $s_n$ such that $n\ell(1/s_n)s_n^{\beta} \to 0$, which is invariant under the scaling transformation $s_n\mapsto \lambda s_n$ for every $\lambda>0$.
\end{proof}
Observe that the large deviation condition $n\,\P[X_1>x_n]\xrightarrow{n\to\infty}0$
 is sharp, being at the boundary with the ``natural'' fluctuations of the GCLT: due to spectral positivity, it coincides with the defining condition of the norming sequence in the convergence~\eqref{eq:rareEvent}.

The process $(S_n)_{n\in\N}$ in Corollary \ref{cor:IID} is the rigorous formulation of what is know by physicist as a L\'evy flight~\cite{Mandelbrot82}: a random walk with i.i.d. regularly varying increments with infinite variance (and possibly infinite mean). This is the simplest application of random walks with regularly varying increments in the literature. 
We now turn to more complex situations.

\medskip
Historically, dependent regularly varying sequences represent a challenging and non-trivial branch of the large deviation theory. In the field of time series analysis for heavy-tailed processes, these kind of models appeared in the form of moving average processes and autoregressive models~\cite[Ch.~7]{EKM}. Consider the usual sequence $(X_n)_{n\in \N}$ of i.i.d. regularly varying random variables, that are commonly referred  as \emph{innovations} or \emph{noise}, and two sets of parameters $\theta_1,\dots,\theta_q$ and $\phi_1,\dots,\phi_p$. We can construct a causal linear process 
\[
S_1\coloneqq X_1,\quad S_2\coloneqq X_2+\theta_1 X_1, \quad\dots \quad S_n\coloneqq X_n+\theta_1X_{n-1}+\dots+\theta_qX_{n-q},\quad n> q,
\]
called \emph{moving average process of order $q$} and denoted by MA($q)$, which represents a $q$-dependent regularly varying sequence.
In the same spirit, the \emph{autoregressive model of order $p$}, denoted as AR($p$), is obtained with a recursive equation
\[
S_n\coloneqq \sum_{i=1}^p \phi_i S_{n-i}+X_n,\quad n\in\N.
\]
In the former, the correlation is given by a direct finite-range propagation of the noise term $X_n$ to $q$ future values in the time series $(S_n)_{n\in\N}$, whereas in the latter $X_{n-1}$ produces an indirect effect on $S_n$ by appearing in the recurrence equation for $S_{n-1}$, and so by iteration it causes the presence of an infinite memory. 

While large deviation estimates for the MA($q$) process are known (refer to~\cite{Mikosch2013}) and can be easily recovered from Proposition~\ref{prop:RVI} (see the discussion after Corollary~\ref{cor:InonID}), less is known for the AR($p$) model (see~\cite{Mikosch2000}). We will now deal with a similar setting built on a memory kernel originally introduced by the physics community~\cite{Chechkin2009, Meerschaert2009, Holl2021}, where our method can be directly applied.
Let  $(S_n)_{n\in\N}$ denote the partial sums whose correlated increments $(Y_k)_{k\in\N}$ are themselves partial sums of an i.i.d. sequence $(X_j)_{j\in\N}$ of $\R^+$-valued random variables in the class $\cR\cV_{-\beta}$ with $\beta \in \brange$, weighted by some positive memory kernel: 
\begin{equation}\label{eq:corrSum}
S_n\coloneqq\sum_{k=1}^n Y_k\quad\text{with}\quad Y_k\coloneqq\sum_{j=1}^k m_{k-j} X_j\quad\text{and}\quad m_0= 1,\,m_{k-j}>0.
\end{equation}
For the sake of simplicity, restrict $(X_j)_{j\in\N}$ to the \emph{normal} domain of attraction of a spectrally-positive $\beta$-stable distribution: the introduction of a slowly varying function can be handled in the same way.  Moreover, we use the same notation as before: $b_n\coloneqq 0$ if $\beta \in (0,1)$  and $b_n\coloneqq  \E[S_n]$ if $\beta \in (1,2)$.
The following corollary follows again directly from Proposition~\ref{prop:RVI}:

\medskip
\begin{corollary}\label{cor:InonID}
Let $(S_n)_{n\in \N}$ be the partial sums of correlated random variables as  in~\eqref{eq:corrSum}, and define $L_n\coloneqq \sum_{k=1}^n \left( \sum_{i=0}^{n-k}m_i\right)^\beta$. Then, for $(x_n)_{n\in\N}$ such that $x_n\to \infty$ and $ L_n\P(X_1> x_n)\to 0$ as $n \to\infty$ we have
\begin{equation}
\lim_{n\to\infty}\sup_{x\geq x_n}\left|\frac{\P [S_n -b_n > x]}{ L_n\,\P[X_1> x]}-1  \right|= 0.
\end{equation}
\end{corollary}
\medskip\noindent
\begin{proof}
Firstly observe that, rearranging the summands,  $S_n$ can be recast as the partial sum of the row-wise independent random variables  of the triangular array $ \tilde{X}_{k,n}\coloneqq w_{n-k} X_k$, where $ w_{n-k}\coloneqq\sum_{i=0}^{n-k}m_i$.
Denote by $\hat F^{(k,n)}(s) = 1 + a_{k,n}(s)$ the Laplace--Stieltjes transform of the distribution of the possibly centered $\tilde{X}_{k,n}$; see~\eqref{eq:b} at the beginning of this section for the precise centering convention. Therefore we can write $a_{k,n}(s)=w_{n-k}^{\beta}a(s)$, with $a(s)=-\Gamma(1-\beta)s^{\beta}+o(s^\beta)$ as  $s\searrow 0$ coming from the transform $\hat{F}(s)=1+a(s)$ of the \textcolor{blue} centered $X_1$. The assumptions of Proposition~\ref{prop:RVI} are  thus readily satisfied by noticing that $A_n(s)=L_na(s)$ with $L_n=\sum_{k=1}^n w_{k-n}^{\beta}$, and taking $s_n$ such that $L_n s_n^{\beta} \to 0$ as $n\to\infty$.
\end{proof}

\medskip
In order to appreciate the generality of Corollary~\ref{cor:InonID}, consider two iconic types of memory functions: the exponential kernel, that represents a short-memory regime with weak correlation, and the algebraic counterpart with long-memory effects. The first one is obtained by setting
\[
m_{k}= \nu^{k} \quad\text{for some}\quad 0<\nu<1.
\]
It is immediate that the following asymptotic holds $w_{n-k}\sim (1-\nu)^{-1}$ as $n\to\infty$ for fixed $k$, which is responsible for an ``almost'' i.i.d. situation: the absence of long-range correlations provides the usual scaling $ L_n\asymp n$ as $n\to \infty$. 
The exponential memory kernel clearly satisfies the recursive equation of an AR($1$) model with $\phi_1\equiv \nu$.
Similarly, in a MA($q$)-process, the approximately i.i.d. condition~\eqref{eq:almostIID} is guaranteed by the even simpler fact that $w_{n-k}\equiv1+\theta_1+\dots+\theta_q$ up to $k=n-q-1$.

\noindent
This observation fails in the presence of stronger correlations, as for instance with the choice
\[
m_{k}= (k+1)^{-\nu}\quad\text{with}\quad 0<\nu<1,
\]
for which $w_{n-k}\sim\frac {n^{1-\nu}}{1-\nu}$ as $n\to\infty$ for fixed $k$, that leads to a superlinear growth $L_n\asymp n^{1+\beta(1-\nu)}$ and to the failure of~\eqref{eq:almostIID}.

\subsection{Sums of a random number of independent increments}
Let $(\mathrm{N}(t))_{t\geq 0}$ denote a stochastic process with values that are nonnegative and integer, independent of the triangular array $((X_{k,n})_{k=1}^n)_{n\in\N}$, and consider the random sums 
\begin{equation}\label{eq:defRS}
S_0\coloneqq 0,\qquad S_{\mathrm{N}(t)}\coloneqq \sum_{k=1}^{\mathrm{N}(t)}X_{k, \mathrm{N}(t)}\quad\text{for}\quad \mathrm{N}(t)\geq 1.
\end{equation}
In several situations of interest a sum of random variables is stopped at  a random time. In particular, one of the firstly studied setting is where the random number of summands  $\mathrm{N}(t)$ is given by a renewal counting process. 
In physics, for instance, $\mathrm{N}(t)$ corresponds to the number of i.i.d. waiting times $(T_k)_{k\in\N}$ occurred up to time $t$ in a (wait-then-jump) Continuous-Time Random Walk model (CTRW)~\cite{Montroll65}, that can be addressed as the random sum associated with a L\'evy flight. 
In finance and insurance, instead, $\mathrm{N}(t)$ is often the  number of claims in the interval $[0,t]$ in an actuarial risk process. 

Large deviations for sums of a random number of independent increments have been derived -- as far as we know~\cite{Kluppelberg97, Tang2001, Ng2004, BollLit2} -- under the fundamental assumption that the random time concentrates with high probability around its expectation $\E[\mathrm{N}(t)]<\infty$, see~\eqref{eq:convProb} below. Indeed, leveraging on such assumption, the random sum is asymptotically dominated by the sum of a fixed number $\E[\mathrm{N}(t)]$ of increments. Our method does not rely on this splitting, allowing to deal with processes for which~\eqref{eq:convProb} fails or even $\E[\mathrm{N}(t)]=\infty$, as we will show below. In particular, we will discuss renewal and compound renewal risk models, highlighting cases where our technique leads to novel results beyond those available in the existing literature.

\medskip
We intend to derive large deviation results for this class of processes via the uniform Tauberian Theorem~\ref{thm:tauberian}. To lighten the discussion,  we will firstly restrict to  the  case of  i.i.d. random variables $(X_k)_{k\in\N}$.  
We will comment on how to generalise the forthcoming analyses to the more general setting in Appendix~\ref{app:RS-triangular}. Note that limit theorems for weighted sums as in Corollary~\ref{cor:InonID} but with a random number of increments have been derived in~\cite[Thm.~2.4]{Gut2011}.

Suppose  that the generating function of the nonnegative, $\N$-valued process $(\mathrm{N}(t))_{t\geq 0}$ is well defined for $z\in[0,z_t)$ for some $z_t\geq 1$ and takes the  form 
$$
\phi_{\mathrm{N}(t)}(z)= \sum_{n=0}^{\infty}\P( \mathrm{N}(t)=n) z^n= 1 + \ell_t(1-z) (1-z)^{\gamma} + o((1-z)^{\gamma}) \quad \textrm{as} \quad z \to 1^-, 
$$
with constant $\gamma>0$ for all $t\geq 0$ and $\ell_t(1-z)$ slowly varying when $z \to 1^-$. 
The Laplace--Stieltjes transform of the centred sum $Z_t\coloneqq S_{\mathrm{N}(t)}-b_{\mathrm{N}(t)}$ (with $b_{\mathrm{N}(t)}$ defined in~\eqref{eq:b}), in a proper right neighborhood of the origin $s\in[0,s_t)$, can be thus written as
\begin{equation}\label{eq:RS}
\hat F_t(s)=\phi_{\mathrm{N}(t)}(\hat F(s)) = 1  + L_t(1/s)\cdot s^{\gamma \beta} + \cE_t(s),
\end{equation} 
where  the slowly varying function at infinity $L_t$ and the sub-leading term $\cE_t(s)$ take contributions both from the  Laplace--Stieltjes transform $\hat F(s)$ of the random variables $(X_k)_{k\in\N}$\,---\,with  $X_1\in \cR\cV_{-\beta}$ centered if $\beta\in(1,2)$\,---\,and the generating function of the $\N$-valued process $\phi_{\mathrm{N}(t)}(z)$. This is the effect of the interplay between the two sources of randomness: in the general setting it is not possible to  disentangle  these contributions  in order to provide proper assumptions  on the spatial and temporal patterns separately.
In order to apply  Theorem~\ref{thm:LDP}  to the random variables $Z_t$ we thus need a joint control over  both terms.
In the following, we discuss  several settings where this fine control on the Laplace--Stieltjes transform can be  achieved.

For the reader's benefit, we recall that by assumption $\hat F(s)=1+a(s)$ with $a(s)=-\Gamma(1-\beta)\ell(1/s)s^\beta+o(\ell(1/s)s^\beta)$ as $s\searrow 0$ and $\P(X_1>x) \sim \ell(x)x^{-\beta}$ as $x\to\infty$, given $\beta \in \brange$ and $\ell(x)$ a positive slowly varying function at infinity. In particular, $a(s)$ is eventually positive when $\beta>1$.

\subsubsection{A couple of explicit examples}
As instructive examples, we start with two stochastic processes $(\mathrm{N}(t))_{t\geq 0}$ whose generating function can be written in closed form, thus allowing the derivation of their large deviation principles via a direct application of Theorem~\ref{thm:LDP}.

\medskip
\begin{example}\label{ex:Poisson}
Let $(T_k)_{k\in\N}$ be a sequence of exponentially distributed i.i.d. waiting times with common tail distribution $e^{-\rho t}$ for some $\rho>0$. 
Then the counting process $\mathrm{N}(t)\coloneqq \sup\{n\geq 0\,:\, T_1+\dots+T_n\leq t\}$ has a Poisson distribution with $\E[\mathrm{N}(t)]=\rho \,t$. From $\phi_{{\mathrm N}(t)}(z)=e^{-\rho \,t(1-z)}$ we directly get that the sub-leading term $\cE_t(s)$ in~\eqref{eq:RS} is actually negligible in every region $s\leq s_t$ with  $(s_t)_{t\geq 0}$ such that $ a(s_t) \cdot t    \to 0$ as $t\to\infty$, that is
\[
\lim_{t\to\infty}\sup_{s\in(0,\lambda s_t]}\left|
\frac{\hat F_t(s) -1}{ \rho\,t\,a(s)}-1 \right| =
\lim_{t\to\infty}\sup_{s\in(0,\lambda s_t]}\left|
\frac{e^{\rho\, t   a(s)} -1}{ \rho \, t \,a(s)}-1 \right| =0 
\]
for all $\lambda>0$. 
We can therefore immediately conclude via Theorem~\ref{thm:LDP} that for $x_t= 1/s_t$
\[
\lim_{t\to\infty}\sup_{x\geq  x_t} \left|\frac{\P[S_{\mathrm{N(t)}}-b_{\mathrm{N(t)}}>x]}{\rho\,  t \, \P [X_1>x]}-1  \right|=0.
\]
\end{example}

\medskip\noindent

\medskip
\begin{example}\label{ex:exp}
 Let $({\mathrm N}(t))_{t\geq 0}$ be a process with geometric distribution parametrized by $1/\rho(t)\in(0,1)$, where $\rho(t)<\infty$ for every $t<\infty$ and $\rho(t)\xrightarrow{t\to\infty}\infty$. 
From the corresponding generating function $\phi_{{\mathrm N}(t)}=z/(\rho(t)-[\rho(t)-1]z)$, we obtain that~\eqref{eq:RS} becomes
\[
\lim_{t\to\infty}\sup_{s\in(0,\lambda s_t]}\left|
\frac{\hat F_t(s) -1}{ \rho(t)\,a(s)}-1 \right| =\lim_{t\to\infty}\sup_{s\in(0,\lambda s_t]}\left| \frac{(\rho(t)-1)a(s)}{1-(\rho(t)-1)a(s)}\right| =0
\]
for every $\lambda>0$ and $(s_t)_{t\geq 0}$ such that $a (s_t)\cdot \rho(t)\to 0$ as $t\to\infty$, and by applying Theorem~\ref{thm:LDP} with $x_t= 1/s_t$
\[
\lim_{t\to\infty}\sup_{x\geq  x_t} \left|\frac{\P[S_{\mathrm{N(t)}}-b_{\mathrm{N(t)}}>x]}{\rho(t) \, \P [X_1>x]}-1  \right|=0.
\]
\end{example}

\medskip
Example~\ref{ex:Poisson} is the classical compound-Poisson risk model, also known as Cram\'er-Lundberg model: it belongs to a well-studied class  of randomly stopped sums for which large deviation results  are already discussed in the literature (see~\cite{Tang2001, Ng2004}).
In contrast, Example~\ref{ex:exp} lies outside the range of cases typically investigated in the literature, where concentration of N(t) around its expected value is a key requirement (see~\eqref{eq:convProb} below). Nevertheless, our approach allows for an easy derivation of its large deviation principle. See also Example~\ref{ex:Renewal} for a broader and significative class of models with this feature that are handled by our method.

\subsubsection{Processes $(\mathrm{N}(t))_{t\geq 0}$ with finite mean}
Consider the class of processes with $\E[\mathrm{N}(t)]\in(0,\infty)$ for every finite $t$.  For this group, Theorem~\ref{thm:LDP} can be recast in the following form:

\medskip
\begin{corollary}\label{cor:RSFM}
Let $(\mathrm{N}(t))_{t\geq0}$ be a nonnegative, $\N$-valued process and $(S_{\mathrm{N}(t)})_{t\geq0}$ the random sums in~\eqref{eq:defRS}. Assume that $\E[\mathrm{N}(t)]<\infty$ for every $t<\infty$, with $E[\mathrm{N}(t)]\to \infty$ as $t \to\infty$, and denote by $(s_t)_{t\geq 0}$, with $s_t\to 0$ as $t\to\infty$, a family such that for all $\lambda>0$
\begin{equation}\label{eq:derivRS}
\lim_{t\to\infty}\left| \frac{\phi'_{\mathrm{N}(t)}(z)\big|_{1+a(\lambda s_t)} }{\E[\mathrm{N}(t)]}-1\right|=0.
\end{equation}
Then, for $x_t=1/s_t$ we have
\begin{equation}\label{eq:thesisFM}
\lim_{t\to \infty}\sup_{x\geq x_t}\left|\frac{\P[S_{\mathrm{N}(t)}-b_{\mathrm{N}(t)}>x]}{\E[\mathrm{N}(t)]\P[X_1>x]}-1 \right|=0.
\end{equation}
\end{corollary}

\begin{proof}
The hypothesis~\eqref{eq:derivRS} guarantees that the radius of convergence of the probability generating function $\phi_{N(t)}$ is greater than $1$ when $\beta>1$, making $\hat F_t(s)$ well-defined. 
By Bernoulli's inequality and the mean value theorem, for some $0<\tilde s<s$ we have
\begin{align}\label{eq:conv}
0\leq \hat F_{t}(s)-1-\E[\mathrm{N}(t)]a(s)= a(s)\left(\phi'_{\mathrm{N}(t)}(z)\big|_{1+a(\tilde s)}-\E[\mathrm{N}(t)]\right).
\end{align}
Since $a(s)=-\Gamma(1-\beta)\ell(1/s)s^\beta+o(\ell(1/s)s^\beta)$ as $s\searrow 0$ with $\ell(x)$ positive slowly varying function at infinity, according to~\eqref{eq:RS} we can write 
\[
\mathcal E_t(s) = \hat F_{t}(s)-1+\Gamma(1-\beta)\E[\mathrm{N}(t)] \ell(1/s)s^\beta \quad\text{as}\quad s\searrow 0.
\] 
Recalling that $\phi_{\mathrm{N}(t)}$ is convex by definition and using~\eqref{eq:conv}, we can therefore apply Theorem~\ref{thm:LDP} observing that for all $\lambda>0$\footnote{In the last line we use the fact that, by the convexity of $\phi_{\mathrm{N}(t)}$ and the monotone equivalents theorem~\cite[Thm.~1.5.3]{Bingham} for $a(s)$, we have:
\begin{itemize}
\item for $\beta\in (0,1)$, i.e.\ $a(s)<0$: $\inf_{s\in(0,\lambda s_t]}\phi'_{\mathrm{N}(t)}(1+a(s)) \sim \phi'_{\mathrm{N}(t)}(1+a(\lambda s_t))$ as $t \to\infty$;
\item for $\beta\in (1,2)$: $\sup_{s\in(0,\lambda s_t]}\phi'_{\mathrm{N}(t)}(1+a(s)) \sim \phi'_{\mathrm{N}(t)}(1+a(\lambda s_t))$ as $t \to\infty$.
\end{itemize}}
\begin{align*}
0\leq\lim_{t\to\infty} &\sup_{s\in(0,\lambda s_t]}\frac{|\mathcal E_t(s)|}{\E[\mathrm{N}(t)]|a(s)|}\\
& \leq\lim_{t\to\infty}\sup_{s\in(0,\lambda s_t]}\left(\left|\frac{\hat F_t(s)-1 }{\E[\mathrm{N}(t)]a(s)}-1\right| 
+ \left| \frac{o(\ell(1/s)s^\beta)}{a(s)}\right| \right)\\
& \leq \lim_{t\to\infty} \begin{cases}
\displaystyle 1- \frac{\inf_{s\in(0,\lambda s_t]}\phi'_{\mathrm{N}(t)}(z)|_{1+a(s)}}{\E[\mathrm{N}(t)]} \qquad & \text{if}\quad \beta\in (0,1)\\
\displaystyle \frac{\sup_{s\in(0,\lambda s_t]}\phi'_{\mathrm{N}(t)}(z)|_{1+a(s)}}{\E[\mathrm{N}(t)]}-1 \qquad & \text{if}\quad \beta\in (1,2)
\end{cases} \\
& = \lim_{t\to\infty}\left|\frac{\phi'_{\mathrm{N}(t)}(z)\big|_{1+a(\lambda s_t)} }{\E[\mathrm{N}(t)]}-1\right|
=0 
\end{align*}
by assumption, from which
\[
\lim_{t\to\infty}\sup_{s\in(0,\lambda s_t]} \frac{|\mathcal E_t(s)|}{\E[\mathrm{N}(t)]\ell(1/s)s^\beta}=0.
\]
The conclusion immediately follows given that $\P(X_1>x) \sim \ell(x)x^{-\beta}$ as $x\to\infty$.
\end{proof}

The hypothesis in Corollary \ref{cor:RSFM} are easily checked on specific cases: we can thus leverage on this reformulation  to identify  explicit conditions on the process $(\mathrm{N}(t))_{t\geq 0}$ for which large deviations  are readily derived for the corresponding randomly stopped sum. Below we give three examples of such conditions. We then discuss on  how these examples complement and compare with the hypotheses  known in literature.
\\

As a first example,  it is not difficult to verify that the following assumption for generic finite-mean processes $(\mathrm{N}(t))_{t\geq0}$ fulfill Corollary~\ref{cor:RSFM}, irrespective of the tail  of the common distribution of the summands $(X_k)_{k\in\N}$:

\medskip
\begin{assumption}\label{Assumption3}
For every $q\in\N$ 
\[
\E[\mathrm{N}^q(t)] =O((\E[\mathrm{N}(t)])^q)\quad \text{as}\quad t \to \infty.
\]
\end{assumption}

\noindent
In fact, by definition 
\begin{align*}
\phi_{\mathrm{N}(t)}(z)-1&
=\sum_{k=1}^\infty(z-1)^k\overbrace{\sum_{n=k}^\infty\binom n k \P[\mathrm{N}(t)=n]}^{\eqqcolon B_k(t)},
\end{align*}
and, by convexity of $\phi_{\mathrm{N}(t)}$, for $\beta>1$ (similarly for $\beta<1$) and $t$ large enough
\begin{align*}
0\leq \frac{\phi'_{\mathrm{N}(t)}(z)\big|_{1+a(s_t)}}{\E[\mathrm{N}(t)]}-1&=\frac 1{\E[\mathrm{N}(t)]}\sum_{k=2}^\infty k B_k(t)[a(s_t)]^{k-1}=O(\E[\mathrm{N}(t)]a(s_t)),
\end{align*}
which is invariant under the scaling transformation $s_t\mapsto \lambda s_t$ for every $\lambda>0$. The hypothesis~\eqref{eq:derivRS} is therefore satisfied by taking $(s_t)_{t\geq 0}$ such that $\E[\mathrm{N}(t)]a(s_t)\to 0$ as $t\to\infty$.

\begin{example}\label{ex:Renewal}
Let us consider the renewal model, that is a popular risk model where the counting process $\mathrm{N}(t)\coloneqq \sup\{n\geq 0\,:\, T_1+\dots+T_n\leq t\}$ (as in Example~\ref{ex:Poisson}) has i.i.d. inter-arrival times $(T_k)_{k\in\N}$ independent of the individual non-negative claim sizes $(X_k)_{k\in\N}$. In the literature~\cite{EKM}\cite[Thm.~2.3--2.4]{Tang2001}, the fundamental assumption is $\E[T_1]<\infty$, which implies $\E[\mathrm{N}(t)]\asymp t$ as $t \to\infty$ and that $\mathrm{N}(t)$ concentrates with high probability around its expectation (see~\eqref{eq:convProb} below). Assumption~2, however, still holds true if we assume that $T_k$'s are in the domain of attraction of a $\zeta$-stable law with $\zeta\in(0,1)$ (namely $\E[T_1]=\infty$): in this case, $\E[\mathrm{N}(t)]\asymp t^\zeta$ is finite for every finite $t$ and similarly for all higher moments, but it is well-known that the weak Law of Large Numbers \eqref{eq:convProb} fails. In fact, the scaled process $\mathrm{N}(t)/t^\zeta$ converges in distribution to a Mittag-Leffler random variable of index $\zeta$ in the renewal theory sense~\cite[\S~3]{MS2004}, i.e. with mean $1$\footnote{Which in turn, consistently with classical renewal theory when $\E[T_1]<\infty$~\cite[Corollary~8.6.2]{Bingham}, converges weakly to a point mass at $1$ when $\zeta\to 1$.} (or equivalently to the inverse of a $\zeta$-stable subordinator).
In conclusion, this example makes explicit that the validity of result~\eqref{eq:thesisFM}  no longer requires the finiteness of $\E[T_1]$. A significant application of the case where all moments are infinite can be found in CTRW models, which have been extensively used in the physics community~\cite{MontrollScher1973, ScherMontroll1975, Shlesinger1974}.
\end{example}

\medskip\noindent
On the other hand, if we restrict ourselves to  random variables  with infinite mean (i.e. $\beta\in(0,1)$ for $X_1\in \cR\cV_{-\beta}$), even weaker conditions on $(\mathrm{N}(t))_{t\geq 0}$ are sufficient.
In the first instance, recalling that $a(s)<0$, by Bernoulli's inequality we can simply set $q=2$ in Assumption~\ref{Assumption3}:
\medskip
\begin{assumption}\label{Assumption3b}
Assume that $X_1\in \cR\cV_{-\beta}$ with $\beta\in(0,1)$ and that $(\mathrm{N}(t))_{t\geq 0}$ is  such that 
\[
\E[\mathrm{N}^2(t)]
=O((\E[\mathrm{N}(t)])^2)\quad \text{as}\quad t \to \infty.
\]
\end{assumption}
\noindent
Indeed, in the limit $t\to\infty$ we get
\begin{align*}
0\leq 1-\frac{\phi'_{\mathrm{N}(t)}(z)\big|_{1+a(s_t)}}{\E[\mathrm{N}(t)]}&\leq \frac{|a(s_t)|}{\E[\mathrm{N}(t)]}\sum_{n=0}^\infty n(n-1)\P[\mathrm{N}(t)=n]\leq |a(s_t)|\frac{\E[\mathrm{N}^2(t)]}{\E[\mathrm{N}(t)]}. 
\end{align*}
Under Assumption~\ref{Assumption3b}, we can choose $(s_t)_{t\geq 0}$ such that $\E[\mathrm{N}(t)]a(s_t)\to 0$ as $t\to\infty$ in order to apply Corollary~\ref{cor:RSFM}.

Alternatively, relying on Corollary~\ref{cor:IID} and on a probabilistic argument, we can consider the following (different) sufficient condition; refer to Appendix~\ref{app:RS} for more details.
\medskip
\begin{assumption}\label{Assumption4} 
Assume that $X_1\in \cR\cV_{-\beta}$ with $\beta\in(0,1)$ and that for every $\delta>0$
\[
\sum_{n>(1+\delta)\E[\mathrm{N}(t)]}n\,\P[\mathrm{N}(t)=n]=o(\E[\mathrm{N}(t)])\quad\text{as}\quad t \to\infty.
\]
\end{assumption}

\medskip
In literature, known results involve nonnegative i.i.d. random variables $(X_k)_{k\in\N}$ with $\beta>1$ and are directly stated \cite{Kluppelberg97, Tang2001, Ng2004} in terms of $ S_{\mathrm{N}(t)}-\E[S_{\mathrm{N}(t)}]$. As pointed out also for convergence in law of random sums~\cite{Shant1984, Korolev1995}~\cite[Thm.~5.1]{Finkel1994}, the convergence in probability
\begin{equation}\label{eq:convProb}
\frac{\mathrm{N}(t)}{\E[\mathrm{N}(t)]}\xrightarrow{p}1\quad \text{as}\quad t \to \infty\quad (\text{with}\quad \E[\mathrm{N}(t)]\xrightarrow{t\to\infty} \infty),
\end{equation}
is a necessary (but not sufficient) hypothesis in order to perform the passage from $b_{\mathrm{N}(t)}$ to a nonrandom centering. 
Note that Example~\ref{ex:exp} and the renewal processes in Example~\ref{ex:Renewal} are not in this class as they do not satisfy this condition\footnote{It is immediate to show that assumption~\eqref{eq:convProb} fails for Example~\ref{ex:exp} since for all $\delta>0$
\begin{align*}
\P[\mathrm{N}(t)>(1+\delta)\E[\mathrm{N}(t)]]&=\frac 1 {\rho(t)}\sum_{n>(1+\delta)\rho(t)}\left(1-\frac 1{\rho(t)}\right)^{n-1}\xrightarrow{t\to\infty}\exp{-(1+\delta)}\neq 0.
\end{align*}}; however they fulfill our Assumption~\ref{Assumption3}.
The hypothesis present in~\cite{Tang2001, Ng2004}, which implies~\eqref{eq:convProb}, is
\begin{equation}\label{eq:Tang}
\forall\, \delta>0\,,\quad \text{for some}\quad \epsilon>0\,,\quad \E[\mathrm{N}^{\beta+\epsilon}(t)\mathds{1}_{\mathrm{N}(t)>(1+\delta )\E[\mathrm{N}(t)]}]=O(\E[\mathrm{N}(t)]),
\end{equation}
with $\beta>1$ by assumption, but the nonrandom centering forces a restriction on the large deviation condition that becomes $x_t= K\E[\mathrm{N}(t)]$, for any $K>0$.

Anyhow, as witnessed by an additional example in Appendix~\ref{app:RS}, not only Assumption~\ref{Assumption3} does not necessarily guarantee the validity of~\eqref{eq:convProb} and it is not embedded in~\eqref{eq:Tang}, but neither the converse implication holds.

\medskip
\begin{remark}
The choice of the random centering $b_{N(t)}$ unifies the statement with the case of an infinite-mean distribution for the increments $(X_k)_{k\in\N}$.
Moreover, it is directly related to the disappearance of the additional asymptotic term in the big-jump equivalence for randomly stopped sums with nonnegative increments~\cite[\S~4]{Fay2006} \cite{Alesk2008,  Robert2008, Denisov2010}. This modified max-sum equivalence for noncentered partial sums under regular variation can be equivalently deduced by applying the Karamata's classical theorem~\cite[Thm.~1.7.1]{Bingham} to~\eqref{eq:RS}.
\end{remark}

\subsubsection{Processes $(\mathrm{N}(t))_{t\geq 0}$ with infinite mean}
As a final aspect, it is worth emphasizing that the hypothesis $\E[\mathrm{N}(t)]<\infty$ as $t<\infty$ is not necessary in order to obtain large deviation results for partial sums of i.i.d. random variables stopped at random time. 
Since we will focus on the case $\E[\mathrm{N}(t)]=\infty$ for $t<\infty$, hereinafter we set $b_{\mathrm{N}(t)}=0$ independently of the value of $\beta$, and consequently $\hat F(s)=1+a(s)$ will denote the Laplace--Stieltjes transform of the distribution function of $X_1$, with $a(s)\in[-1,0]$.
As a direct application of Theorem~\ref{thm:LDP}, we have:
\begin{corollary}\label{cor:RSIM}
Let $(\mathrm{N}(t))_{t\geq 0}$ be a nonnegative, $\N$-valued process such that there exist two positive sequences $(C_t)_{t\geq 0}$ and $(z_t)_{t\geq 0}$, with $z_t\to1$ as $t\to\infty$, for which $\forall\lambda>0$
\begin{equation}\label{eq:Ct}
\phi_{\mathrm{N}(t)}(z)=1-C_t(1-z)^\gamma+\cE_t(1-z), \quad \text{with}\quad \lim_{t\to\infty}\sup_{1-z\leq \lambda (1-z_t)}\frac{\cE(1-z)}{C_t(1-z)^\gamma}=0.
\end{equation}
Let $(S_{\mathrm{N}(t)})_{t\geq0}$ be the random sums in~\eqref{eq:defRS} and $\beta\in(0,1)$. Then, for $x_t\to \infty$ as $t\to\infty$ such that $\P[X_1>x_t]\leq K(1-z_t)$ for some $K>0$, we have
\begin{equation}\label{eq:MikRS}
\lim_{t\to \infty}\sup_{x\geq x_t}\left|\frac{\P[S_{\mathrm{N}(t)}>x]}{C(\gamma,\beta)\P[\bar F(x) \mathrm{N}(t)>1]}-1 \right|=0,
\end{equation}
where $C(\gamma,\beta)\coloneqq\frac{\Gamma(1-\gamma)[\Gamma(1-\beta)]^\gamma}{\Gamma(1-\gamma\beta)}$.
\end{corollary}

\begin{proof}
Recall that by hypothesis $\P[\mathrm{N}(t)>n]\sim C_t n^{-\gamma}/\Gamma(1-\gamma)$ as $n\to\infty$, and $\P(X_1>x)\sim -a(1/x)/\Gamma(1-\beta)$ as $x\to\infty$. By assumption, for all $\lambda>0$
\begin{align*}
\lim_{t\to\infty}\sup_{s\in(0,\lambda s_t]}\left|\frac{1-\hat F_t(s) }{C_t|a(s)|^\gamma}-1 \right|
=\lim_{t\to\infty}\sup_{s\in(0,\lambda s_t]}\left|
\frac{1-\hat F_t(s)}{C_t[\Gamma(1-\beta)\ell(1/s)s^\beta]^\gamma}-1 \right|
=0
\end{align*}
if $s_t\to 0$ as $t\to\infty$ is such that $|a(s_t)|\leq K(1-z_t)$ for some $K>0$. Hence, by Theorem~\ref{thm:LDP}, we obtain
\[
\lim_{t\to\infty} \sup_{x\geq1/s_t}\left|\frac{\Gamma(1-\gamma\beta)}{[\Gamma(1-\beta)]^\gamma}\frac{\P[S_{\mathrm{N}(t)}>x]}{C_t[\bar F(x)]^\gamma}-1 \right| =0.
\]
\end{proof}

\begin{remark}
Observe that, although for fixed $t$ we have~\cite[Lemma~5.1]{JessenMikosch2006} $\P[\bar F(x)\mathrm{N}(t)>1]\sim\P[\max_{1\leq k\leq \mathrm{N}(t)}X_k>x]$, it is important to emphasize that, when $\E[\mathrm{N}(t)]=\infty$, the random sum $S_{\mathrm{N}(t)}$ is not dominated by a single big-jump phenomenon. Instead, the asymptotic behaviour is driven by an accumulation of moderate increments, similar to what occurs in the example derived from Corollary~\ref{cor:InonID} with exponential kernel~\cite{Holl2021}. This is closely linked to the inequality $C(\gamma,\beta)>1$; see also~\cite{CP2025} for further discussion. 
\end{remark}

\begin{remark}\label{rmk:FM}
In the case $\beta\in(1,2)$, without the random centering, we obtain a trivial result determined by the law of large numbers (see also~\cite[\S~4]{Fay2006} for the equivalent statement when $t$ is fixed)
\[
\lim_{t\to\infty} \sup_{x\geq x_t}\left|\frac{\P[S_{\mathrm{N}(t)}>x]}{\P[\E[X_1]\mathrm{N}(t)>x]}-1 \right| =0.
\]
Note that, since $\E[\mathrm{N}(t)]=\infty$ for finite $t$, the generating function $\phi_{\mathrm{N}(t)}(z)$ is not well-defined for $z>1$ -- the relevant value for the Laplace transform of the centered random variables $X_k$'s. This technical issue, which affects both asymptotic analysis and large deviation estimates, can be circumvented by using the characteristic function instead; see~\cite{CP2025}. 
\end{remark}

\medskip
By way of illustration, we apply our method to derive large deviation results for the compound renewal model, under assumptions weaker than those used in ~\cite{Tang2001}.
Let us firstly briefly recall the model: let $(S_{\mathrm{N}(t)})_{t\geq 0}$ denote the total claim amount process given by
\[
S_0\coloneqq 0,\qquad S_{ \mathrm{N}(t)}\coloneqq \sum_{k=1}^{\mathrm{N}(t)}X_k\quad \text{where}\quad  \mathrm{N}(t)\coloneqq \sum_{j=1}^{\tau(t)} V_j,
\]
with $\tau(t)\coloneqq \sup\{n\geq 0\,:\, T_1+\dots+T_n\leq t\}$ as in Example~\ref{ex:Renewal}, and $(V_j)_{j\in\N}$ nonnegative, integer-valued i.i.d. random variables. The latter represent the number of individual claims caused by the same occurrence: if $V_1\equiv 1$, we recover the classical renewal process.  Observe that in~\cite[Thm.~2.3--2.4]{Tang2001} both $T_1$ and $V_1$ must have finite mean, losing the possibility to obtain $\E[ \mathrm{N}(t)]=\infty$ for $t<\infty$, that is $\gamma<1$ in~\eqref{eq:RS}.
Before going further, let us stress that, in contrast with the renewal processes in Example~\ref{ex:Renewal}, when both $T_1$ and $V_1$ have finite mean Assumption~\ref{Assumption3} might fail and the generating function $\phi_{\mathrm{N}(t)}(z)$ could exhibit non-analytic behaviour at $z=1$. The same considerations outlined in Remark~\ref{rmk:FM} apply in this scenario; further results will be recovered in~\cite{CP2025}.
\begin{example}\label{ex:CR}
Let $\tau(t)$ be the Poisson process with rate $\rho>0$ as in Example~\ref{ex:Poisson}, and let $V_1$ be a nonnegative, integer-valued random variable such that 
\[
\P[V_1=k]=\frac 1 {\zeta(1+\gamma)}k^{-1-\gamma},\qquad k\in\N,
\]
where the normalization constant $\zeta(1+\gamma)$ denotes the Riemann zeta function.  By an integral approximation, it is immediate to realize that $V_1$ has a power law tail of index $\gamma\in(0,1)$, namely $P[V_1>n]=\sum_{k=n+1}^\infty \P[V_1=k]\approx n^{-\gamma}/[\gamma \zeta(1+\gamma)]$ as $n\to \infty$. 
We get\footnote{Note that  in the limit $z\to 1^-$
\[
\mathrm{Li}_{1+\gamma}(z)=\zeta(1+\gamma)-|\Gamma(-\gamma)|(1-z)^\gamma(1+o(1))+O(1-z),
\]
given that $\log(1/z)=\sum_{n=1}^\infty (1-z)^n/n$ and~\cite[Eq.~(25.12.12)]{NIST}
\[
\mathrm{Li}_s(z)=\zeta(s)+\Gamma(1-s)\left(\log (1/z)\right)^{s-1}+\sum_{n=1}^\infty\zeta(s-n)\frac{(\log(z))^n}{n!},\qquad s\neq 1,2,3,\dots,\quad |\log z|<2\pi.
\]}
\[
\phi_{\mathrm{N}(t)}(z)=\exp{\rho t(\phi_{V_1}(z)-1)},\quad \phi_{V_1}(z)=\sum_{k=1}^\infty \P[Y=k]z^k=
\frac {\mathrm{Li}_{1+\gamma}(z)}{\zeta(1+\gamma)}\quad \text{if}\quad z\in (0,1),
\]
which directly implies that the sub-leading term $\cE_t(s)$ in~\eqref{eq:RS} is actually negligible in every region $s\in(0, s_t]$ with  $(s_t)_{t\geq 0}$ such that $ t\cdot |a(s_t)|^\gamma \to 0$ as $t\to\infty$, that is
\[
\lim_{t\to\infty}\sup_{s\in(0,\lambda s_t]}\left|\frac{\zeta(1+\gamma)}{ |\Gamma(-\gamma)|}
\frac{1-\hat F_t(s)}{\rho\,t\,|a(s)|^\gamma}-1 \right| 
=
\lim_{t\to\infty}\sup_{s\in(0,\lambda s_t]}\left|
\frac{\zeta(1+\gamma)}{ |\Gamma(-\gamma)|}\frac{1-e^{\rho\, t   (\mathrm{Li}_{1+\gamma}(1+a(s))-1)} }{ \rho \, t \,|a(s)|^\gamma}-1 \right| =0 
\]
for all $\lambda>0$. 
We can therefore immediately conclude via Theorem~\ref{thm:LDP} that 
\begin{itemize}
\item if  $\beta\in (0,1)$ and $x_t$ is such that $t\,\ell^\gamma(x_t) x_t^{-\gamma\beta}\to 0$ as $t\to\infty$, we get
\[ 
\lim_{t\to\infty}\sup_{x\geq  x_t} \left|\frac{\zeta(1+\gamma)\Gamma(1-\gamma\beta)}{ |\Gamma(-\gamma)|[\Gamma(1-\beta)]^\gamma}\frac{\P[S_{\mathrm{N}(t)}-b_{\mathrm{N}(t)}>x]}{\rho\,  t \, \ell^\gamma(x) x^{-\gamma\beta}}-1  \right|=0;
\]
\item if $\beta\in(1,2)$ and $x_t$ is such that $t/x_t^{\gamma}\to 0$ as $t \to\infty$, we have
\[ 
\lim_{t\to\infty}\sup_{x\geq  x_t} \left|\frac{\gamma\zeta(1+\gamma)}{\rho\,\ E[X_1]^\gamma}\frac{\P[S_{\mathrm{N}(t)}-b_{\mathrm{N}(t)}>x]}{  t \, x^{-\gamma}}-1  \right|=0.
\]
\end{itemize}
\end{example}
\noindent
It is clear that this example can be easily generalized. $\mathrm{N}(t)$ belongs to the class described in Example~\ref{ex:Renewal}. By choosing generic $V_j$'s with tail decay exponent $\gamma\in (0,1)$ and $\tau(t)$, we always rely on the fact that, as a consequence of Corollary~\ref{cor:RSFM}, there exist $(C_t)_{t\geq 0}$ in~\eqref{eq:Ct} that is equal to $\E[\tau(t)]$ except for a constant factor that does not depend on $t$, yet is related to the tail constant of $V_1$.

\medskip
For a different example, in a similar spirit to~\cite[Sec.~4]{Gut2011}, consider a stopped random walk $(S_{\mathrm{N}(t)})_{t\geq 0}$ with i.i.d. increments $(X_k)_{k\in\N}$ where $\mathrm{N}(t)$ is provided by some cost at the first passage time $\tau$ to the positive semi-axis (see~\cite{Artuso2014, Bianchi2022}) of a second (symmetric and continuous) random walk $(\tilde S_n)_{n\in\N}$ on $\R$ (with $\tilde S_0=0$) independent of the first one. More precisely, denote by $(V_k(t))_{k\in \N}$ the sequence, dependent on $t$, of $\N$-valued costs associated with the \emph{length} of each jump of the underlying random walk $(\tilde S_n)_{n\in\N}$. We assume that $(V_k(t))_{k\in \N}$ are i.i.d. with $\E[V_1(t)]<\infty$ for all $t<\infty$, and $\E[V_1(t)]\to \infty$ as $t\to \infty$. As a consequence of the \emph{generalized Spitzer-Baxter identity}~\cite[Cor.~3.4]{Bianchi2022}, we know that 
\[
\mathrm{N}(t)=\sum_{k=1}^\tau V_k(t),\qquad \phi_{\mathrm{N}(t)}(z)=1-\sqrt{1-\phi_{V_1(t)}(z)}.
\]
It is worth emphasizing that, in this second class of examples, the roles of $\tau$ and $V_1$ are reversed compared to the case of compound renewal processes. Specifically, we assume that $V_1(t)$ has a finite first moment for all $t\geq 0$, while
$\E[\tau]=\infty$ since $\tau$ now refers to a first-passage event rather than to a counting process.

\medskip
\begin{corollary}
Let $(\mathrm{N}(t))_{t\geq0}$ be the $\N$-valued process described above, with a cost variable $V_1(t)$ such that $ \E ( [V_1(t)]^q ) =O((\E[ V_1(t) ] )^q)\, \,\text{as}\,\, t \to \infty$, and let $(S_{\mathrm{N}(t)})_{t\geq0}$ be the random sums in~\eqref{eq:defRS}. Then 
\\
\begin{itemize}
\item if  $\beta\in (0,1)$  and $x_t$ is such that $\E[V_1(t)]\P[X_1>x_t]\to0$, we have 

\[
\lim_{t\to \infty}\sup_{x\geq x_t}\left| \frac{\Gamma(1-\beta/2)}{\sqrt{\Gamma(1-\beta)}}\frac{\P[S_{\mathrm{N}(t)}>x]}{\sqrt{\E[V_1(t)]|\P[X_1>x]|}}-1 \right|=0\,;
\]
\item if  $\beta\in (1,2)$ and $x_t= K \E[V_1(t)]$ as $t\to\infty$ for any $K>0$, we have
\[
\lim_{t\to \infty}\sup_{x\geq x_t}\left| \frac{\P[S_{\mathrm{N}(t)}>x]}{\P\left[\mathrm{N}(t)>\frac x{\E[X_1]}\right]   }-1 \right|=0\,.
\]
\end{itemize}
\end{corollary}

\begin{proof}
It is sufficient to observe that
\begin{align*}
\hat F_t(s)&=1-\sqrt{1-\sum_{n=0}^\infty\sum_{k=0}^n \binom n k [a(s)]^k \P[V_1(t)=n]}\\
&=1-\sqrt{-\E[V_1(t)]a(s)}\sqrt{1+\sum_{k=2}^\infty [a(s)]^{k-1} \frac{C_k(t)}{\E[V_1(t)]}},
\end{align*}
where $C_k(t)\coloneqq \sum_{n=k}\binom n k \P[V_1(t)=n]$. Recalling that $a(s)<0$, exploiting  the assumption on the moments of $V_1(t)$,  and choosing $s_t$ such that $\E[V_1(t)]a(s_t)\xrightarrow{t\to\infty}0$, we obtain that, given $\epsilon>0$, for $t$ large enough
\[
\left|\frac{1-\hat F_t(\tilde s_t)}{\sqrt{-\E[V_1(t)]a(\tilde s_t)}}-1\right|<\epsilon\quad\text{for every}\quad \tilde s_t\leq s_t.
\]
Hence, we can directly conclude by applying Theorem~\ref{thm:LDP}.
\end{proof}

\section{Proofs}\label{sec:proofs}
\subsection{Proof of Theorem~\ref{thm:tauberian}}
We extend the argument used in \cite{Karamata1930} for the original version of Tauberian results; see also \cite[\S\S 7.6,7.11]{Hardy} and \cite[\S V.4]{Widder}.
\begin{proof}[Proof for $\alpha>0$]
Define the function $j:\R^+\to [0,\mathrm e]$ as
\[
j(x)\coloneqq \begin{cases}
   0 & \quad \textrm{for}  \quad x\in[0,1/\mathrm e), \\
   1/x & \quad \textrm{for}  \quad x\in[1/\mathrm e,+\infty) ,
\end{cases}
\]
and note that, given $\lim_{x\to-\infty}G_t(x)=0$ and $s>0$, we have
\begin{equation}\label{eq:Gwithj}
G_t(1/s)=\int_{-\infty}^{\infty} \exp{-sx}j(\exp{-sx})\df G_t(x).
\end{equation}

\medskip
\begin{claim*} For all $\epsilon>0$, there exist polynomials  $P_{\epsilon}(x)$ and $p_{\epsilon}(x)$  such that
\begin{equation}\label{eq:Weier1}
p_{\epsilon}(x)\le j(x)\le P_{\epsilon}(x)\, , \quad \quad \forall x\in[0,+\infty),
\end{equation}
and
\begin{equation}\label{eq:Weier2}
\int_0^{\infty} \exp{-x} x^{\alpha-1} \left[P_{\epsilon}(\exp{-x})-p_{\epsilon}(\exp{-x}) \right]\df x \le \epsilon \Gamma(\alpha).
\end{equation}
\end{claim*}

\medskip\noindent
Accepting this claim for the time being, that will follow from an adaptation of Weierstrass approximation theorem on compact intervals, we proceed with the proof of the theorem.
From \eqref{eq:Weier1}, using \eqref{eq:Gwithj} and the fact that $G_t(x)$ is non-decreasing in $x$ we have
\begin{subequations}
\begin{equation}\label{eq:ineq1}
\int_{-\infty}^{\infty}  \exp{-sx} p_{\epsilon}(\exp{-sx})\df G_t(x) \le G_t(1/s) \le \int_{-\infty}^{\infty}  \exp{-sx} P_{\epsilon}(\exp{-sx})\df G_t(x),
\end{equation}
\begin{equation}\label{eq:ineq2}
\alpha \int_0^{\infty}  \exp{-x}\,  x^{\alpha-1}\, p_{\epsilon}(\exp{-x}) \df x  \, \le 1 \le \, \alpha\int_0^{\infty}  \exp{-x}\,x^{\alpha-1}\, P_{\epsilon}(\exp{-x})\df x ,
 \end{equation}
\end{subequations}
where we started from the identity $\alpha\int_0^\infty e^{-x}x^{\alpha-1}j(e^{-x})dx=1$ 
to get the second line.
In the following, we shall assume throughout that $t \ge t_0$ and $s \le 1/x_0$, with $t_0$ and $x_0$ from Definition~\ref{def:SVIP}.
Combining~\eqref{eq:ineq1}, divided by $L_t(1/s)s^{-\alpha}$, with~\eqref{eq:ineq2} we get
\[
\left| \frac{G_t(1/s)}{L_t(1/s)s^{-\alpha}} \, - \, 1\right| \le \max\{A_{s,\epsilon},B_{s,\epsilon}  \}\,,
\]
where
\begin{align*}
A_{s,\epsilon}&\coloneqq\left| \frac{\int_{-\infty}^{\infty}  \exp{-sx} P_{\epsilon}(\exp{-sx})\df G_t(x)}{L_t(1/s)s^{-\alpha}} 	- \alpha \int_0^{\infty}\exp{-x}x^{\alpha-1}p_{\epsilon}(\exp{-x})\df x \right|,\\
B_{s,\epsilon}&\coloneqq 
\left| \frac{\int_{-\infty}^{\infty}  \exp{-sx} p_{\epsilon}(\exp{-sx})\df G_t(x)}{L_t(1/s)s^{-\alpha}}- \alpha\int_0^{\infty}\exp{-x}x^{\alpha-1}P_{\epsilon}(\exp{-x})\df x \right|.
\end{align*}
Let $D(\epsilon),d(\epsilon)$ denote the maximum degree and $(a_k(\epsilon))_{k=0}^{D(\epsilon)},(b_k(\epsilon))_{k=0}^{d(\epsilon)}$ the coefficients of the polynomials $P_\epsilon(x),p_\epsilon(x)$ respectively. 
Denote $Dd(\epsilon)\coloneqq\max\{D(\epsilon),d(\epsilon)\}$ and let $t_0(\epsilon)\geq t_0$ be such that $(Dd(\epsilon)+1)s_t< \xi_t$ for all $t\geq t_0(\epsilon)$. For $t>t_0(\epsilon)$ and $s\leq s_t$, we can therefore rewrite
 \begin{align}
 \int_{-\infty}^{\infty}  \exp{-sx} P_{\epsilon}(\exp{-sx})\df G_t(x)&=\sum_{k=0}^{D(\epsilon)}a_k(\epsilon)\hat G_t[(k+1)s],\nonumber\\
 \alpha \int_0^{\infty}\exp{-x}x^{\alpha-1}P_{\epsilon}(\exp{-x})\df x&=\Gamma(\alpha+1)\sum_{k=0}^{D(\epsilon)}a_k(\epsilon)(k+1)^{-\alpha},\label{eq:failAlpha0}
 \end{align}
and similarly for $p_\epsilon(x)$, which jointly with \eqref{eq:Weier2} provide the bounds
\begin{align*}
A_{s,\epsilon} & \leq  \Gamma(\alpha+1)
\sum_{k=0}^{D(\epsilon)}|a_k(\epsilon)|  (k+1)^{-\alpha} \left|\frac{\hat G_t[ (k+1)s]}{\Gamma(\alpha+1) L_t(1/s) [(k+1)s)]^{-\alpha}}-1\right| +\epsilon\Gamma(\alpha+1),\\
B_{s,\epsilon} & \leq \Gamma(\alpha+1) 
\sum_{k=0}^{d(\epsilon)}|b_k(\epsilon)|  (k+1)^{-\alpha} \left|\frac{\hat G_t[ (k+1)s]}{\Gamma(\alpha+1) L_t(1/s) [(k+1)s)]^{-\alpha}}-1\right| +\epsilon\Gamma(\alpha+1).
\end{align*}
Observe that for every $k\in\{0,\dots,Dd(\epsilon)\}$ we can write 
\begin{multline*}
\left| \frac{\hat G_t [(k+1) s]}{\Gamma(\alpha+1) L_t(1/s) [(k+1)s]^{-\alpha}}-1\right|\\ \leq
\left| \frac{L_t(1/[(k+1)s] ) }{ L_t(1/s) }\right| \left| \frac{\hat G_t [(k+1)s]}{\Gamma(\alpha+1) L_t(1/[(k+1)s]) [(k+1)s]^{-\alpha}}-1\right|  + \left| \frac{L_t(1/[(k+1)s] ) }{ L_t(1/s) }-1\right|.
\end{multline*}
Given $\delta>0$, fix $\epsilon =\epsilon(\delta)< \delta/{(3\Gamma(\alpha+1))}$, and consequently  the polynomials $P_{\epsilon},\,p_{\epsilon}$; then define
 $
K(\epsilon)\coloneqq \Gamma(\alpha+1) \max_k\{|a_k(\epsilon)|,|b_k(\epsilon)|\}.
$
Using Definition~\ref{def:SVIP}, we  can find  $\bar t(\delta)\coloneqq \max_k \bar t(\delta,k)$ and $\bar s(\delta)\coloneqq \min_k \bar s(\delta,k)$ such that $\forall \,t> \bar t(\delta)$
\[
\sup_{s\in(0, \bar s(\delta)]}\left| \frac{L_t(1/[(k+1) s] ) }{ L_t(1/s) } -1\right|   \leq \frac{\delta}{ 3 K(\epsilon) (Dd(\epsilon)+1)  }\qquad  \textrm{for every} \quad 0\leq k \leq Dd(\epsilon)\,.
\]
Let $t_1(\delta)$ be such that $t_1(\delta)>\max\{t_0(\delta),\bar t(\delta)\}$ and $ s_t \le \bar s(\delta)$ for all $ t \ge t_1(\delta)$. We can similarly find $t_2(\delta)$ such that $t_2(\delta)> t_1(\delta)$ and, by assumption~\eqref{eq:HypTaub}, $\forall\, t \geq t_2(\delta)$
\begin{multline*}
\sup_{s\in(0,s_t]} \left|\frac{\hat G_t[ (k+1)s]}{\Gamma(\alpha+1) L_t(1/[(k+1)s]) [(k+1)s)]^{-\alpha}}-1\right| \\
< \frac{\delta }{3 K(\epsilon)(Dd(\epsilon)+1)} \, \frac1{1+\delta /[3K(\epsilon)(Dd(\epsilon)+1)]}
 \end{multline*}
for every $k \leq Dd(\epsilon)$.
This gives that 
\[
\sup_{s\in(0, s_t]} \max\{A_{s,\epsilon(\delta)},B_{s,\epsilon(\delta)}\} \le \delta\quad\text{for all} \quad t>t_2(\delta),
\]
hence the proof will be complete once we have proved the \emph{Claim}.

\begin{proof}[Proof of the Claim]
Fix $\epsilon>0$. By the Weierstrass approximation theorem, we can find two polynomials $H_{\epsilon}(x)$ and $h_{\epsilon}(x)$  such that
 \[
h_{\epsilon}(x)\le j(x)\le H_{\epsilon}(x) , \quad \quad \forall  x\in[0,2]\,,
\]
and
\[
\int_0^{\infty} \exp{-x} x^{\alpha-1} \left[H_{\epsilon}(\exp{-x})-h_{\epsilon}(\exp{-x}) \right]\df x \le \frac{\epsilon}2 \Gamma(\alpha);
\]
see~\cite[Thm.~109]{Hardy} or~\cite[Lem.~4.1]{Widder}.
We will now show that there exist two integer numbers $M(\epsilon),m(\epsilon)$ such that the polynomials
\begin{align*}
P_{\epsilon}(x)&\coloneqq H_{\epsilon}(x) + [x(x-1)]^{2M(\epsilon)}\\
p_{\epsilon}(x)&\coloneqq h_{\epsilon}(x) - [x(x-1)]^{2m(\epsilon)}
\end{align*}
satisfy \eqref{eq:Weier1} and \eqref{eq:Weier2}.

The second requirement  \eqref{eq:Weier2} is accomplished  by simply choosing $M(\epsilon)$ and $m(\epsilon)$ large enough so that $[x(x-1)]^{2\min\{M(\epsilon),m(\epsilon)\}}< \epsilon /4 $ for $x \in [0,1]$, that is 
$ \min\{M,m\}> -(\log{\epsilon}-\log 4)/\log 16$.

Concerning the first requirement \eqref{eq:Weier1}, instead, let us consider $P_\epsilon(x)$: the proof is similar for the lower bound $p_{\epsilon}(x)$. Firstly note that, since $H_{\epsilon}(x) \ge j(x)$  for $x \in [0,2] $ and $[x(x-1)]^{2M} \geq 0$ for every $x$, we just need to check that  $P_\epsilon (x) \ge j(x) $ for $x \in [2,\infty)$. A sufficient condition is choosing $M(\epsilon)$ in such a way that $P_{\epsilon}(x) \ge H_{\epsilon}(2)$ for all $x \ge 2$, noting that
\[
[x(x-1)]^{2M} +H_{\epsilon}(x) -H_{\epsilon}(2)  \geq [x(x-1)]^{2M} - a^- n^-x^{M^-} \,,
\]
where $n^-$ is the number of monomials in $H_{\epsilon}(x)-H_{\epsilon}(2)$ with negative coefficients,  $a^-$ is the largest modulus of such negative coefficients, and  $M^-$ is the largest exponent of such monomials. 
Hence, given that the left-hand side in
\[
x^{2M-M^-}(x-1)^{2M}>a^-n^-
\]
is always greater than $2^{2M-M^-}$, the inequality certainly holds true for 
\[
M>\frac{M^-}2+\frac{\log(a^- n^-)}{\log 4}.
\]
Similarly, it is sufficient to choose $m(\epsilon)$ such that $p_\epsilon(x)\leq 0$ for all $x\geq 2$.
\end{proof}

And that concludes the proof of the theorem.
\end{proof}

\begin{proof}[Proof for $\alpha=0$]
Consider the \emph{continuous} auxiliary function $\mathfrak j :\R^+\to [0,1]$ defined by
\[
\mathfrak j(x)\coloneqq \begin{cases}
   0 & \quad \textrm{for}  \quad x\in[0,1/\mathrm e), \\
   \frac 1 x \left(1-\log\frac 1x\right)& \quad \textrm{for}  \quad x\in[1/\mathrm e,+\infty);
\end{cases}
\]
Using the same argument of the claim in the proof for $\alpha>0$ and continuity, we can show that there exist two polynomials such that~\eqref{eq:Weier1} remains unchanged and~\eqref{eq:Weier2} holds pointwise, that is
\[
p_\epsilon(x)\leq \mathfrak j(x)\leq P_\epsilon(x)<p_\epsilon(x)+\epsilon, \qquad \forall\, x\in[0,1].
\]
Moreover, \eqref{eq:Gwithj}~is replaced by
\[
s\int_{-\infty}^{1/s}G_t(x)\df x=\int_{-\infty}^\infty \exp{-sx}\mathfrak j(\exp{-sx})\df G_t(x),
\]
where we used the fact that $\lim_{x\to-\infty}G_t(x)=0$ and the existence of the bilateral Laplace--Stieltjes transform in integration by parts.
We can therefore directly compute
\[
\left| \frac{\int_{-\infty}^{1/s}G_t(x)dx}{L_t(1/s)s^{-1}} \, - \, 1\right| \le \max\{A_{s,\epsilon},B_{s,\epsilon}  \}\,,
\]
where
\begin{align*}
A_{s,\epsilon}&\coloneqq\left| \frac{\int_{-\infty}^{\infty}  \exp{-sx} P_{\epsilon}(\exp{-sx})\df G_t(x)}{L_t(1/s)} 	- 1 \right|\\
&\leq \sum_{k=0}^{D(\epsilon)}|a_k(\epsilon)|  \left|\frac{\hat G_t[ (k+1)s]}{ L_t(1/[(k+1)s])}-1\right| \frac{ L_t(1/[(k+1)s])}{L_t(1/s)} \\
&\hspace{4cm}
+  \sum_{k=0}^{D(\epsilon)}|a_k(\epsilon)|\left|\frac{ L_t(1/[(k+1)s])}{L_t(1/s)}-1\right|+|P_\epsilon(1)-1|,
\end{align*}
and similarly for $B_{s,\epsilon}$. Note that $\mathfrak j(1)=1$. Hence, observing that $|P_\epsilon(1)-\mathfrak j(1)|<\epsilon$ by construction, we can proceed exactly in the same way as in the proof for $\alpha>0$.

Note that, by assumption~\eqref{eq:HypTaub}, we can replace $s_t$ with $s_t/(1+\eta)$ for any $\eta>0$.
Thus, we can apply Lemma~\ref{lem:FTL} (see also Remark~\ref{rmk:FTLweak} for the hypothesis $(L_t)_{t\geq 0}\in\cL_0^+$), and conclude that
\[
\lim_{t\to\infty} \sup_{x\geq 1/s_t}\left|\frac{G_t(x)}{L_t(x)}-1 \right|=0.
\]
\end{proof}

\subsection{Proof of Theorem~\ref{thm:LDP}}\label{proof2}
\begin{proof}
We split the proof according to the value of the parameter $\beta$.

\paragraph{Case $\beta\in(0,1)$}
We apply Theorem~\ref{thm:tauberian} with $\alpha\coloneqq 1-\beta$ to\footnote{By using integration by parts for Lebesgue--Stieltjes integrals~\cite[Thm.~A6.1]{Bingham}\cite[Thm.~2.6.11]{Shiryayev}, with respect to right-continuous functions when $F_t(x)$ has mass at $x=0$.}
\[
\hat G_t(s)\coloneqq \frac{1-\hat F_t(s)}s=\int_0^\infty \exp{-sx}\bar F_t(x)\df x\eqqcolon \int_0^\infty \exp{-sx}\df G_t(x) \qquad s>0.
\]
Hypotheses of Theorem~\ref{thm:tauberian} are satisfied: $G_t(x)\coloneqq 0$ for $x\leq 0$ and $G_t(x)\coloneqq \int_0^x \bar F_t(y)\df y$ for $x>0$ is non-decreasing  as  $\bar F_t(x)=\P[Z_t>x]$ is non-negative.
Hence, noting that the choice of the family $(s_t)_{t\geq 0}$ is defined only up to multiplication by a positive constant, Theorem~\ref{thm:tauberian} yields
\[
\lim_{t\to\infty}\sup_{x\geq C x_t}\left| \frac{(1-\beta) G_t(x)}{L_t(x) x^{1-\beta}}-1\right|=0\quad \text{for all}\quad C>0.
\]
Finally,  to infer from the uniform control on $G_t$  a uniform control over the tail distribution $\bar F_t$, which is non-increasing by definition, we apply Lemma~\ref{lem:FTL} with $C=1/(1+\eta)$; see Section~\ref{subsec:lem}.

\paragraph{Case $\beta\in(1,2)$}
Denote by $\theta$ the Heaviside step function and consider
\[
\hat G_t(s)\coloneqq \frac{\hat F_t(s)-1}{s^2}=\int_{-\infty}^\infty \exp{-sx}\left[FF_t(x)-x\theta(x)\right]\df x \qquad s>0,
\]
where $FF_t(x)\coloneqq\int_{-\infty}^x  F_t(y)\df y$. Note that we used the existence of the bilateral Laplace--Stieltjes transform in integration by parts.\footnote{Or right-continuity if $F_t(x)$ is left-bounded.}
In order to apply Theorem~\ref{thm:tauberian}, we have to check that $G_t(x)$ is non-decreasing in $x$: the hypothesis $\lim_{x\to-\infty}G_t(x)=0$ is clearly satisfied.
Observe that  $\E[Z_t]=\int_0^\infty[1-F_t(x)]\df x -\int_{-\infty}^0F_t(x)\df x=0$ and thus
\begin{align*}
FF_t(x)& =x-\E[Z_t]+\int_x^\infty \bar F_t(y) dy\\
&=x+O(L_t(x)x^{-\beta+1}) \qquad \textrm{as} \quad x \to \infty;
\end{align*}
refer to~\cite[Thm.~2.6.1 and Thm.~2.6.5]{Ibragimov}.
Since $FF_t$ is a convex function, by Lemma~\ref{lem:asympt} in Section~\ref{subsec:lem} we conclude that $G_t$ is monotonically non-decreasing, as required.

\noindent
Lastly, since both 
\[
\frac {\df G_t(x)}{\df x} = FF_t(x)-x\theta(x)\quad \text{and}\quad \frac {\df^2 G_t(x)}{\df x^2} = F_t(x)-\theta(x)
\] 
are ultimately monotone for $x\geq 0$ by definition of $\bar F_t$, the statement is proved by applying Lemma~\ref{lem:FTL} twice.\footnote{First, with $u_t(x)\coloneqq FF_t(x)-x$ for $x\in\R^+$ and $\gamma\coloneqq 2-\beta\in(0,1)$; second, with $u_t(x)\coloneqq -\bar F_t(x)$ for $x\in\R^+$ and $\gamma\coloneqq 1-\beta<0$.}
\end{proof}

\subsection{Auxiliary lemmata}\label{subsec:lem}
In order to apply our main result to random walk models, we  need to transfer uniform asymptotic estimates under differentiation. We extend a classical result on differentiating asymptotic relations~\cite[Thm.~1.7.2 and Thm.~1.7.5]{Bingham}. 

\medskip
\begin{lemma}\label{lem:FTL} 
Let $(u_t(x))_{t\geq 0}$ be a one-parameter family of locally bounded functions that are non-increasing on $\R^+$ and such that $\lim_{x\to-\infty}u_t (x)=0$.
Suppose that there exist $\gamma\in (0,1)$ and $(x_t)_{t\geq 0}$, with $x_t\xrightarrow{t\to\infty}\infty$, such that
\[
\lim_{t\to\infty}\sup_{x\geq x_t}\left|\frac{\int_{-\infty}^x u_t(y)\df y}{L_t(x)x^{\gamma}}-1 \right| =0,
\]
where $(L_t)_{t\geq 0}\in\cL^+_\text{sep}$. Then
\[
\lim_{t\to\infty}\sup_{x\geq (1+\eta)x_t}\left|\frac{u_t(x)}{\gamma L_t(x)x^{\gamma-1}}-1 \right| =0\qquad \forall\, \eta>0.
\]
Moreover, the same conclusion holds for a family of functions $(u_t(x))_{t\geq 0}$ non-decreasing on $\R^+$ with $\gamma\in(-\infty,0)\bigcup[1,\infty)$.
\end{lemma}

\begin{proof}
Denote $U_t(x)\coloneqq\int_{-\infty}^x u_t(y)\df y$.
By hypothesis, $\forall\,\epsilon>0$ there exists $\tilde t=\tilde t(\epsilon)$ such that for all $t>\tilde t$ 
\[
\sup_{x\geq x_t}\left|\frac{U_t(x)}{L_t(x)x^{\gamma}}-1 \right|<\epsilon^2.
\]
Using Definition~\ref{def:SVIP} with $\eta=\epsilon^2$, which provides $\bar t(\epsilon,\Lambda)$ and $\bar x(\epsilon,\Lambda)$, we can find  $\cT=\cT(\epsilon, \Lambda)>\max\{\tilde t(\epsilon),\bar t(\epsilon,\Lambda)\}$  such that $x_t\geq \bar x$ for every $t>\cT$.

Without loss of generality assume that $\gamma\in (0,1)$ and, for all $t\geq 0$, $u_t(x)$ is non-increasing. 
By monotonicity
\begin{equation}\label{eq:mon}
U_t(x+\epsilon x)-U_t(x)=\int_x^{x+\epsilon x}u_t(y)\df y 
\leq\epsilon x \,u_t(x)\leq \int_{x-\epsilon x}^x u_t(y) \df y=U_t(x)-U_t(x-\epsilon x).
\end{equation}
As a consequence, 
\[ 
\left|\frac{u_t(x)}{\gamma L_t(x)x^{\gamma-1}}-1 \right| \leq\max\{A,B\},
\]
where
\[
A\coloneqq \left|\frac{U_t(x+\epsilon x)-U_t(x)}{\epsilon \gamma L_t(x)x^\gamma}-1\right|,\qquad B\coloneqq \left|\frac{U_t(x)-U_t(x-\epsilon x)}{\epsilon\gamma L_t(x)x^\gamma}-1\right|.
\]
Let $\eta>0$. Observe that for every $t>\cT(\epsilon,\Lambda=1+\epsilon)$
\begin{align*}
\sup_{x\geq (1+\eta)x_t}A&\leq\sup_{x\geq x_t}A\\
&\leq \frac 1{\gamma \epsilon}\sup_{x\geq x_t}\left|\frac{U_t(x)}{ L_t(x)x^\gamma}-1\right|\\
&\hspace{1cm}+\frac{(1+\epsilon)^\gamma}{\gamma\epsilon} \sup_{x\geq x_t}\left|\frac{U_t(x+\epsilon x)}{ L_t(x+\epsilon x)(x+\epsilon x)^\gamma}\right|\sup_{x\geq x_t}\left|\frac{L_t((1+\epsilon)x)}{L_t(x)}-1\right|\\
&\hspace{1cm}+\frac{(1+\epsilon)^\gamma}{\gamma\epsilon} \sup_{x\geq x_t}\left|\frac{U_t(x+\epsilon x)}{ L_t(x+\epsilon x)(x+\epsilon x)^\gamma}-1\right|
+\frac 1 {\gamma \epsilon}|(1+\epsilon)^\gamma-1-\gamma\epsilon|\\
&< \frac\epsilon\gamma+\frac{(1+\epsilon)^\gamma}{\gamma}(2+\epsilon^2)\epsilon+O(\epsilon)=O(\epsilon).
\end{align*}
Similarly, taking $\epsilon<\eta/(1+\eta)$, for all $t>\cT(\epsilon,\Lambda=1-\epsilon)$ we have 
\begin{align*}
\sup_{x\geq (1+\eta)x_t}B&\leq\sup_{x\geq \frac{x_t}{1-\epsilon} }B\\
&\leq \frac 1{\gamma \epsilon}\sup_{x\geq x_t}\left|\frac{U_t(x)}{ L_t(x)x^\gamma}-1\right|\\
&\hspace{1cm}+\frac{(1-\epsilon)^\gamma}{\gamma\epsilon} \sup_{x\geq \frac{x_t}{1-\epsilon}}\left|\frac{U_t(x-\epsilon x)}{ L_t(x-\epsilon x)(x-\epsilon x)^\gamma}\right|\sup_{x\geq x_t}\left|\frac{L_t((1-\epsilon)x)}{L_t(x)}-1\right|\\
&\hspace{1cm}+\frac{(1-\epsilon)^\gamma}{\gamma\epsilon} \sup_{x\geq\frac{x_t}{1-\epsilon}}\left|\frac{U_t(x-\epsilon x)}{ L_t(x-\epsilon x)(x-\epsilon x)^\gamma}-1\right|
+\frac 1 {\gamma \epsilon}|1-(1-\epsilon)^\gamma-\gamma\epsilon|\\
&< \frac\epsilon{\gamma}+\frac{(1-\epsilon)^\gamma}{\gamma}(2+\epsilon^2)\epsilon+O(\epsilon)=O(\epsilon).
\end{align*}
The desired conclusion follows by taking the limit $\epsilon \to 0$.
\end{proof}
\medskip
\begin{remark}\label{rmk:FTLweak}
Note that, in Lemma~\ref{lem:FTL}, the assumption that $(L_t)_{t \ge 0} \in \mathcal{L}^+_\text{sep}$ can be relaxed to $(L_t)_{t \ge 0} \in \mathcal{L}^+_{\gamma-1}$ for $\gamma \ge 1$ (and similarly for $\gamma < 0$).
\end{remark}

\begin{remark}\label{rmk:FTL}
Observe that the monotonicity assumption for $u_t(x)$ can be weakened to a ``ultimately monotone'' hypothesis by keeping track of the eventual $t$-dependence of the ``ultimately'' in~\eqref{eq:mon} with respect to the growth rate of the sequence $(x_t)_{t\geq 0}$. This is also related to the motivation for Definition~\ref{def:SVIP}. 
\end{remark}

\medskip \noindent
Lastly, the following lemmata are elementary facts that we state for the sake of readability.
\medskip
\begin{lemma}\label{lem:asympt}
Suppose that $F:\R\mapsto \R$ is convex and has an asymptote $y=ax+b$, with $a,b\in\R$, for $x\to +\infty$. Then $F(x)-ax-b\geq 0$ for every $x\in\R$.
\end{lemma}

\begin{proof}
The statement is an immediate consequence of convexity.
\end{proof}

\begin{lemma}\label{lem:eqPointUnif}
Let $(h_t(s))_{t\geq 0}$ be a family of functions where $h_t:(0,\infty)\to[0,\infty)$ for each $t\geq  0$.  Let $(s_t)_{t\geq 0}$ be such that $s_t \to 0$ as $t\to \infty$. The following statements are equivalent:
\begin{enumerate}[label=(\roman*)]
\item 
$ \displaystyle
\lim_{t\to\infty} \sup_{0<s\leq s_t} h_t(s)=0.
$
\item For any $(\tilde s_t)_{t\geq  0}$ such that $\tilde s_t\leq s_t$ for all $t\geq 0$, we have
$ \displaystyle
\lim_{t\to\infty}h_t(\tilde s_t)=0.
$
\end{enumerate}
\end{lemma}
\begin{proof}
We prove the two implications.
\paragraph{$(i)\implies (ii)$} Let $(\tilde s_t)_{t\geq 0}$ be such that $\tilde s_t\leq s_t$ for all $t\geq  0$. Then, for each $t$, we have $0\leq h_t(\tilde s_t)\leq \sup_{0<s\leq s_t} h_t(s)$, which implies  $ \lim_{t\to\infty} h_t(\tilde s_t)=0$.
\paragraph{$(ii)\implies (i)$}  By contradiction, suppose that there exist $\epsilon>0$ and a subsequence $(t_n)_{n\in\N}$, with $t_n\to \infty$ as $n\to\infty$, such that $\sup_{0<s\leq s_{t_n}} h_{t_n}(s)\geq 2\epsilon$. 
Hence, there exists a sequence $(\bar s_n)_{n\in\N}$ such that $\bar s_n\leq  s_{t_n}$ and $h_{t_n}(\bar s_n)\geq \epsilon$. Define $\tilde s_t\coloneqq \bar s_n$ if $t=t_n$, $\tilde s_t\coloneqq s_t$ otherwise.
Then $\tilde s_t\leq s_t$ for all $t$, but $h_{t_n}(\tilde s_{t_n})
\geq \epsilon$, so $\limsup_{t\to\infty} h_t(\tilde s_t)\geq \epsilon$, contradicting the assumption~$(ii)$.
\end{proof}


\bigskip
{\bf Acknowledgements}
We thank Alessandra Bianchi and Francesco Caravenna for useful discussions. 
This research is part of the authors' activity within the UMI Group ``DinAmicI'' (www.dinamici.org).
The research is part of GC's activity within GNFM/INdAM.
The research of GP was partly supported by the CY Initiative of Excellence through the grant Investissements d'Avenir ANR-16-IDEX-0008, and was partly done under the auspices of the GNFM while GP was a post-doctoral researcher at University of Milano-Bicocca (Milan, Italy).

\bigskip

\appendix
\section{Random sums -- further examples}\label{app:RS}
The following example, together with Example~\ref{ex:exp}, shows that~\eqref{eq:Tang} and Assumption~\ref{Assumption3} have a non-trivial intersection.

\medskip
\begin{example}
Suppose that $\mathrm{N}(t)$ has discrete probability distribution
\[
\left(1-\frac 1 {\rho^2(t)}\right)\delta_{\mathrm{N}(t),\rho(t)}+\frac 1 {\rho^2(t)}\delta_{\mathrm{N}(t),\lfloor\rho^{1+1/\gamma}(t)\rfloor},
\]
with $\gamma\in(1,2)$ such that $\gamma>\beta$ and $\rho(t)\xrightarrow{t\to\infty}\infty$.
It is easy to show that $\E[\mathrm{N}(t)]=\rho(t)$, \eqref{eq:convProb} holds, $\E[\mathrm{N}^\gamma(t)]=o(\rho(t))$ but Assumption~\ref{Assumption3} fails for $q>2\gamma$.
\end{example}

\bigskip
In the main text, we also stated that when $\beta\in(0,1)$ we can obtain a large deviation result under Assumption~\ref{Assumption4}, which implies~\eqref{eq:convProb} and is implied by~\eqref{eq:Tang}; see~\cite[Lem.~3.3 and Eq.~(10)]{Tang2001}. 

\noindent
Given $\delta>0$ and denoting $p_n(t)\coloneqq \P[\mathrm{N}(t)=n]$, by Bernoulli's inequality we can write
\begin{multline*}
0\leq\hat F_t(s)-1-\E[\mathrm{N}(t)]a(s)=\sum_{n< (1+\delta)\E[\mathrm{N}(t)]}[\hat F_n(s)-1-n a(s)]\,p_n(t)\\
+a(s)\left( \sum_{n< (1+\delta)E[\mathrm{N}(t)]}np_n(t)-\E[\mathrm{N}(t)]\right)
+\sum_{n\geq (1+\delta)\E[\mathrm{N}(t)]}\underbrace{[\hat F_n(s)-1]}_{<0}\,p_n,
\end{multline*}
and consequently
\begin{align*}
\left| \frac{\hat F_t(s)-1}{\E[\mathrm{N}(t)]a(s)}-1\right|\leq\sum_{n<(1+\delta)\E[\mathrm{N}(t)]}\frac{|\hat F_n(s) -1-na(s)|}{n|a(s)|}\, \frac {n p_n(t)} {\E[\mathrm{N}(t)]}+\frac{\E[\mathrm{N}(t)\mathds 1_{n\geq (1+\delta)\E[\mathrm{N}(t)]}] }{\E[\mathrm{N}(t)]}.
\end{align*}
Considering $s\in(0, s_t]$ with $s_t$ such that $\E[\mathrm{N}(t)]a(s_t)\xrightarrow{t\to\infty}0$, we can conclude by means of Corollary~\ref{cor:IID} and Assumption~\ref{Assumption4}, which allow us to get the uniform control needed to apply Theorem~\ref{thm:LDP}.

\section{Random sums -- triangular array of independent, not identically distributed increments}\label{app:RS-triangular}
As is the case with Proposition~\ref{prop:RVI}, the large deviation result in Corollary~\ref{cor:RSFM} can be extended to random sums with independent but not identically distributed summands.
We always leverage the notation introduced at the beginning of Section~\ref{sec:Appl}.
Assume that there exist $\beta\in\brange$, $(L_t)_{t\geq 0}\in\mathcal L_\text{sep}^+$ and $s_t\searrow0$ such that $\forall\, \lambda>0$ 
\[
\sup_{s\in(0,\lambda s_t]} \left|\frac{\sum_{n=0}^\infty A_n(s)p_n(t)}{\Gamma(1-\beta)L_t(1/s)s^{\beta}} \right| \longrightarrow 1 \quad \textrm{as} \quad t\to \infty,
\]
recalling that $p_n(t)\coloneqq \P[\mathrm{N}(t)=n]$, and observe that as for Assumption~\ref{assumptionRVI2}
\[
\max_{k=1,\dots, n} |a_{k,n}(s)|= O\left(\frac {|A_n(s)|}n\right).
\]
Since 
\begin{align*}
\hat F_t(s)&=\sum_{n=0}^\infty (1+a_{1,n}(s))(1+a_{2,n}(s))\dots (1+a_{n,n}(s)) p_n(t)\\
&=1+\sum_{n=0}^\infty A_n(s)p_n(t)+\sum_{n=0}^\infty \sideset{}{'}\sum_{k=2}^n a_{j_1,n}(s)a_{j_2,n}(s)\dots a_{j_k,n}(s) p_n(t),
\end{align*}
where the primed sum denotes the sum over the $k$-combinations from the set of the coefficients $a_{k,n}(s)$, only minor changes are required: rephrase Assumption~\ref{Assumption3} in terms of $A_n$. This result can be compared with~\cite{BollLit2}, that is the natural extension of the results~\cite{Tang2001, Ng2004} for i.i.d. increments.

 \bibliographystyle{plain}
\bibliography{biblio}

@book{Bingham,
  author		= "Bingham, N. H. and Goldie, C. M. and Teugels, J. L.",
  title			= "Regular Variation",
  address		= "Cambridge",
  publisher		= "Cambridge University Press",
  year			= "1987"
}

@book{Petrov,
  author		= "Petrov, V. V.",
  title			= "Sums of Independent Random Variables",
  address		= "Berlin",
  publisher		= "Springer",
  year			= "1975"
}

@book{Shiryayev,
  author		= "Shiryayev, A. N.",
  title			= "Probability",
  address		= "Berlin",
  publisher		= "Springer-Verlag",
  year			= "1984"
}

@book{NIST,
  author		= "Olver, F. W. J. and Lozier, D. W. and Boisvert, R. F. and Clark, C. W.",
  title			= "NIST Handbook of Mathematical Functions.",
  address		= "New York",
  publisher		= "Cambridge University Press",
  year			= "2010"
}

@article{MS2004,
  author		= "Meerschaert, M. M. and Scheffler, H-P.",
  title			= "Limit theorems for continuous-time random walks with infinite mean waiting times",
  journal		= "J. Appl. Prob.",
  volume		= "41",
  pages		= "623--638",
  year			= "2004"
}

@article{Karamata1930,
  author		= "Karamata, J.",
  title			= "{\"U}ber die {H}ardy--{L}ittlewoodschen {U}mkehrungen des {A}belschen {S}tetigkeitssatzes",
  journal		= "Math. {Z.}",
  volume		= "32",
  pages			= "319--320",
  year			= "1930",
  doi			= "10.1007/BF01194636"
}

@book{Hardy,
  author		= "Hardy, G. H.",
  title			= "Divergent Series",
  address		= "Oxford",
  publisher		= "Oxford University Press",
  year			= "1949"
}

@book{Widder,
  author		= "Widder, D. V.",
  title			= "The Laplace Transform",
  address		= "Princeton",
  publisher		= "Princeton University Press",
  year			= "1946"
}

@article{Cramer1938,
  author		= "Cram{\'e}r, H.",
  title			= "Sur un nouveau th{\'e}or{\`e}me-limite de la th{\'e}orie des probabilit{\'e}s",
  journal		= "Actualit{\'e}s {S}ci. {I}ndust.",
  volume		= "736",
  pages			= "2--23",
  year			= "1938"
}

@article{MontrollScher1973,
  author		= "Montroll, E. W. and Scher, H.",
  title			= "Random Walks on Lattices. {IV}. {C}ontinuous-Time Walks and Influence of Absorbing Boundaries",
  journal		= "J. Stat. Phys.",
  volume		= "9",
  pages		= "101--135",
  year			= "1973"
}

@article{ScherMontroll1975,
  author		= "Scher, H. and Montroll, E. W.",
  title			= "Anomalous transit-time dispersion in amorphous solids",
  journal		= "Phys. Rev. B",
  volume		= "12",
  pages		= "2455--2477",
  year			= "1975"
}

@article{Shlesinger1974,
  author		= "Shlesinger, M. F.",
  title			= "Asymptotic Solutions of Continuous-Time Random Walks",
  journal		= "J. Stat. Phys.",
  volume		= "10",
  pages		= "421--434",
  year			= "1974"
}

@book{Dembo,
  author		= "Dembo, A. and Zeitouni, O.",
  title			= "Large {D}eviations {T}echniques and {A}pplications",
  address		= "New York",
  publisher		= "Springer",
  year			= "1998"
}

@book{DenHollander,
  author		= "{den} {H}ollander, F.",
  title			= "Large {D}eviations",
  address		= "Providence",
  publisher		= " American Mathematical Soc.",
  year			= "2000"
}

@book{Ellis,
  author		= "Ellis, R. S.",
  title			= "Entropy, {L}arge {D}eviations, and {S}tatistical {M}echanics",
  address		= "New York",
  publisher		= "Springer",
  year			= "1985"
}

@article{Gartner77,
  author		= "G{\"a}rtner, J.",
  title			= "On large deviations from the invariant measure",
  journal		= "Th. Prob. Appl.",
  volume		= "22",
  pages		= "24--39",
  year			= "1977"
}

@article{Ellis84,
  author		= "Ellis, R. S.",
  title			= "Large deviations for a general class of random vectors",
  journal		= "Ann. Probab.",
  volume		= "12",
  pages		= "1--12",
  year			= "1984"
}

@article{Embrechts80,
  author		= "Embrechts, P. and Goldie, C. M.",
  title			= "On closure and factorization properties of subexponential and related distributions",
  journal		= "J. Austral. Math. Soc.",
  volume		= "29",
  pages		= "243--256",
  year			= "1980"
}

@article{Chistyakov64,
  author		= "Chistyakov, V. P.",
  title			= "A theorem on sums of independent positive random variables and its applications to branching random processes",
  journal		= "Theory Probab. Appl. ",
  volume		= "9",
  pages		= "640--648",
  year			= "1964"
}

@article{Lai76,
  author		= "Lai, T. L.",
  title			= "Uniform {T}auberian theorems and their applications to renewal theory and first passage problems",
  journal		= "Ann. Probab. ",
  volume		= "4",
  pages		= "628--643",
  year			= "1976"
}

@article{Heyde67a,
  author		= "Heyde, C. C.",
  title			= "A contribution to the theory of large deviations for sums of independent random variables",
  journal		= "Z. Wahr. verw. Geb.",
  volume		= "7",
  pages		= "303--308",
  year			= "1967"
}

@article{Heyde67b,
  author		= "Heyde, C. C.",
  title			= "On large deviation problems for sums of random variables which are not attracted to the normal law",
  journal		= "Ann. Math. Stat.",
  volume		= "38",
  pages		= "1575--1578",
  year			= "1967"
}

@article{Heyde68,
  author		= "Heyde, C. C.",
  title			= "On large deviation probabilities in the case of attraction to a non--normal stable law",
  journal		= "Sankhya: Indian J. Stat.",
  volume		= "30",
  pages		= "253--258",
  year			= "1968"
}

@article{Nagaev65,
  author		= "Nagaev, S. V.",
  title			= "Some Limit Theorems for Large Deviations",
  journal		= "Theor. Prob. Appl.",
  volume		= "10",
  pages		= "214--235",
  year			= "1965",
  doi			= "10.1137/1110027"
}

@article{Nagaev69a,
  author		= "Nagaev, A. V.",
  title			= "Integral limit theorems taking large deviations into account when {C}ram{\'e}r's condition does not hold. {I}",
  journal		= "Theor. Prob. Appl.",
  volume		= "14",
  pages		= "51--64",
  year			= "1969"
}

@article{Nagaev69b,
  author		= "Nagaev, A. V.",
  title			= "Integral limit theorems taking large deviations into account when {C}ram{\'e}r's condition does not hold. {II}",
  journal		= "Theor. Prob. Appl.",
  volume		= "14",
  pages		= "193--208",
  year			= "1969"
}

@article{Nagaev79,
  author		= "Nagaev, A. V.",
  title			= "Large deviations of sums of independent random variables",
  journal		= "Ann. Prob.",
  volume		= "7",
  pages		= "745--789",
  year			= "1979"
}

@article{Nagaev82,
  author		= "Nagaev, S. V.",
  title			= "On the Asymptotic Behavior of One-Sided Large Deviation Probabilities",
  journal		= "Theory Prob. Appl.",
  volume		= "26",
  pages		= "362--366",
  year			= "1982",
  doi			= "10.1137/1126035 "
}

@article{Mikosch98,
  author		= "Mikosch, T. and Nagaev, A. V.",
  title			= "Large {D}eviations of {H}eavy--{T}ailed {S}ums with {A}pplications in {I}nsurance",
  journal		= "Extremes",
  volume		= "1",
  pages		= "81--110",
  year			= "1998"
}

@article{Ng2004,
  author		= "Ng, K.W. and Tang, Q. and Yan, J.--A. and Yang, H.",
  title			= "Precise large deviations for sums of random variables with consistently varying tails",
  journal		= "J. Appl. Prob.",
  volume		= "41",
  pages		= "93--107",
  year			= "2004"
}

@book{Ibragimov,
  author		= "Ibragimov, I. A. and Linnik, Y. V.",
  title			= "Independent and stationary sequences of random variables",
  address		= "Groningen",
  publisher		= "Wolters--Noordhoff",
  year			= "1971"
}

@article{Kluppelberg97,
  author		= "Kl{\"u}ppelberg, C. and Mikosch, T.",
  title			= "Large deviations of heavy-tailed random sums with applications in insurance and finance",
  journal		= "J. Appl. Prob.",
  volume		= "34",
  pages		= "293--308",
  year			= "1997"
}

@article{Tang2001,
  author		= "Tang, Q. and Su, C. and Jiang, T. and Zhang, J.",
  title			= "Large deviations for heavy-tailed random sums
in compound renewal model",
  journal		= "Stat. Probab. Lett.",
  volume		= "52",
  pages		= "91--100",
  year			= "2001"
}

@article{Denisov2010,
  author		= "Denisov, D. and Foss, S. and Korshunov, D.",
  title			= "Asymptotics of randomly stopped sums in
the presence of heavy tails",
  journal		= "Bernoulli",
  volume		= "16",
  pages		= "971--994",
  year			= "2010",
 doi			= "10.3150/10-BEJ251"
}

@article{Robert2008,
  author		= "Robert, C. Y. and Segers, J.",
  title			= "Tails of random sums of a heavy-tailed number of light-tailed terms",
  journal		= "Insur. Math. Econ.",
  volume		= "43",
  pages		= "85--92",
  year			= "2008"
}

@article{Alesk2008,
  author		= "Ale{\v s}kevi{\v c}ien{\'e}, A. and Leipus, R. and {\v S}iaulys, J.",
  title			= "Tail behavior of random sums under consistent variation with applications to the compound renewal risk model",
  journal		= "Extremes",
  volume		= "11",
  pages		= "261--279",
  year			= "2008",
 doi			= "10.1007/s10687-008-0057-3"
}

@book{EKM,
  author		= "Embrechts, P. and Kl{\"u}ppelberg, C. and Mikosch, T.",
  title			= "Modelling Extremal Events: For Insurance and Finance",
  address		= "Berlin, Heidelberg",
  publisher		= "Springer-Verlag",
  year			= "1997"
}

@article{Chechkin2009,
  author		= "Chechkin, A. V. and Hofmann, M. and Sokolov, I. M.",
  title			= "Continuous-time random walk with correlated waiting times",
  journal		= "Phys. Rev. E",
  volume		= "80",
  pages		= "031112",
  year			= "2009",
 doi			= "10.1103/PhysRevE.80.031112"
}

@article{Meerschaert2009,
  author		= "Meerschaert, M. M. and Nane, E. and Xiao, Y.",
  title			= "Correlated continuous time random walks",
  journal		= "Stat. Probab. Lett.",
  volume		= "79",
  pages		= "1194--1202",
  year			= "2009",
 doi			= "10.1016/j.spl.2009.01.007"
}

@article{Holl2021,
  author		= "Holl, M. and Barkai, E.",
  title			= "Big jump principle for heavy-tailed random walks with
correlated increments",
  journal		= "Eur. Phys. J. B",
  volume		= "94",
  pages		= "216",
  year			= "2021",
 doi			= "10.1140/epjb/s10051-021-00215-7"
}

@article{Mikosch2013,
  author		= "Mikosch, T. and Wintenberger, O.",
  title			= "Precise large deviations for dependent regularly varying
sequences",
  journal		= "Probab. Theory Relat. Fields",
  volume		= "156",
  pages		= "851--887",
  year			= "2013",
 doi			= "10.1007/s00440-012-0445-0"
}

@article{Denisov2008,
  author		= "Denisov, D. and Dieker, A. B. and Shneer, V.",
  title			= "Large deviations for random walks under subexponentiality: the big-jump domain",
  journal		= "Ann. Probab.",
  volume		= "36",
  pages		= "1946--1991",
  year			= "2008",
 doi			= "10.1214/07-AOP382"
}

@article{Esscher32,
  author		= "Esscher, F.",
  title			= "On the probability function in the collective theory of risk",
  journal		= "Skand. Aktuarietidskr",
  volume		= "15",
  pages		= "175--195",
  year			= "1932"
}

@article{Zamparo2023,
  author		= "Zamparo, M.",
  title			= "Large deviation principles for renewal--reward processes",
  journal		= "Stoch. Process. Their Appl.",
  volume		= "156",
  pages		= "126--245",
  year			= "2023"
}

@article{Greven94,
  author		= "Greven, A. and {den} Hollander, F.",
  title			= "Large deviations for a random walk in a random environment",
  journal		= "Ann. Probab.",
  volume		= "22",
  pages		= "1381--1428",
  year			= "1994"
}

@article{Gartner99,
  author		= "G{\"a}rtner, J. and {den} Hollander, F.",
  title			= "Correlation structure of intermittency in the parabolic {A}nderson model",
  journal		= "Probab. Theory Relat. Fields ",
  volume		= "114",
  pages		= "1--54",
  year			= "1999"
}

@article{Pakes75,
  author		= "Pakes, A. G.",
  title			= "On the tails of waiting--time distributions",
  journal		= "J. Appl. Prob.",
  volume		= "12",
  pages		= "555--564",
  year			= "1975"
}

@book{Resnick,
  author		= "Resnick, S. I.",
  title			= "Heavy--Tail Phenomena. Probabilistic and Statistical Modeling",
  address		= "New York",
  publisher		= "Springer",
  year			= "2006"
}

@book{Kulik,
  author		= "Kulik, R. and Soulier, P.",
  title			= "Heavy--Tailed Time Series", 
  address		= "New York",
  publisher		= "Springer",
  year			= "2020"
}

@book{FKZ,
  author		= "Foss, S. and Korshunov, D. and Zachary, S.",
  title			= "An Introduction to Heavy--Tailed and Subexponential Distributions", 
  address		= "New York",
  publisher		= "Springer",
  year			= "2013"
}

@book{Mandelbrot82,
  author		= "Mandelbrot, B.",
  title			= "The fractal geometry of nature", 
  address		= "San Francisco",
  publisher		= "Freeman",
  year			= "1982"
}

@article{Montroll65,
  author		= "Montroll, E. W. and Weiss, G. H.",
  title			= "Random walks on lattices {II}",
  journal		= "J. Math. Phys.",
  volume		= "6",
  pages		= "167--181",
  year			= "1965"
}

@article{Shlesinger82,
  author		= "Shlesinger, M. F. and Klafter, J. and Wong, Y. M.",
  title			= "Random walks with infinite spatial and temporal moments",
  journal		= "J. Stat. Phys.",
  volume		= "27",
  pages		= "499--512",
  year			= "1982",
  doi			= "10.1007/BF01011089"
}

@article{Klafter87,
  author		= "Klafter, J. and Blumen, A. and Shlesinger, M. F.",
  title			= "Stochastic pathway to anomalous diffusion",
  journal		= "Phys. Rev. A",
  volume		= "35",
  pages		= "3081",
  year			= "1987"
}

@article{Fay2006,
  author		= "F{\"a}y, G. and Gonz{\'a}lez-Ar{\'e}valo, B. and Mikosch, T. and Samorodnitsky, G.",
  title			= "Modeling teletraffic arrivals by a {P}oisson cluster process",
  journal		= "Queueing Syst.",
  volume		= "54",
  pages		= "121--140",
  year			= "2006",
  doi			= "10.1007/s11134-006-9348-z"
}

@article{Shant1984,  
  author		= "Shanthikumar, J. G. and Sumita, U.",
  title			= "A central limit theorem for random sums of random variables",
  journal		= "Oper. Res. Lett.",
  volume		= "3",
  pages		= "153--155",
  year			= "1984",
  doi			= "10.1016/0167-6377(84)90008-7"
}

@article{Korolev1995,  
  author		= "Korolev, V. Yu.",
  title			= "Convergence of Random Sequences with the Independent Random Indices {I}",
  journal		= "Theory Probab. Appl.",
  volume		= "39",
  pages		= "282--297",
  year			= "1995",
  doi			= "10.1137/1139018"
}

@article{Finkel1994,  
  author		= "Finkelstein, M. and Kruglov, V. M. and Tucker, H. G.",
  title			= "Convergence in Law of Random Sums with Non-Random Centering",
  journal		= "J. Theor. Probab.",
  volume		= "7",
  pages		= "565--598",
  year			= "1994"
}

@article{Gut2011,  
  author		= "Gut, A.",
  title			= "Ancombe's Theorem 60 years later",
  journal		= "Seq. Anal.",
  volume		= "31",
  pages		= "368--396",
  year			= "2012",
  doi			= "10.1080/07474946.2012.694349"
}

@article{BollLit2,  
  author		= "Sku{\v{c}}ait{\. e}, A.",
  title			= " Large Deviations for Sums of Independent Heavy-Tailed Random Variables",
  journal		= "Lith. Math. J.",
  volume		= "44",
  pages		= "198--208",
  year			= "2004",
  doi			= "10.1023/B:LIMA.0000033784.64716.74"
}

@article{BollLit1,  
  author		= "Paulauskas, V. and Sku{\v{c}}ait{\. e}, A.",
  title			= "Asimptotic results for one-sided large deviation probabilities",
  journal		= "Lith. Math. J.",
  volume		= "43",
  pages		= "318--326",
  year			= "2003",
  doi			= "10.1023/A:1026145503719"
}

@article{Artuso2014,  
  author		= "Artuso, R. and Cristadoro, G. and Degli Esposti, M. and Knight, G.",
  title			= "Sparre--{A}ndersen theorem with spatiotemporal correlations",
  journal		= "Phys. Rev. E",
  volume		= "89",
  pages		= "052111",
  year			= "2014",
  doi			= "10.1103/PhysRevE.89.052111"
}

@article{CP2025, 
  author		= "Cristadoro, G. and Pozzoli, G.",
  title			= "In preparation."
}

@article{JessenMikosch2006, 
  author		= "Jessen, A. H. and Mikosch, T.",
  title			= "Regularly varying functions.", 
  journal		= "Publ. Inst. Math.",
  volume		= "80",
  number		= "94",
  pages		= "171--192",
  year			= "2006",
}

@article{Bianchi2022, 
  author		= "Bianchi, A. and Cristadoro, G. and Pozzoli, G.",
  title			= "Ladder costs for random walks in {L}\'evy random media", 
  journal		= "Stoch. Process. Their Appl.",
  volume		= "188",
  pages		= "104666",
  year			= "2025",
}

@article{ClineHsing,  
  author		= "Cline, D. B. H. and Hsing, T.",
  title			= "Large deviation probabilities for sums and maxima of random variables with heavy or subexponential tails",
  year			= "1991",
  note			= "Preprint, Texas A\&M University"
}

@article{Bingham2008,  
  author		= "Bingham, N. H.",
  title			= "Tauberian theorems and large deviations",
  journal		= "Stochastics",
  volume		= "80",
  pages		= "143--149",
  year			= "2008",
  doi			= "10.1080/17442500701830365"
}

@article{Konst2005,  
  author		= "Konstantinides, D. G. and Mikosch, T.",
  title			= "Large deviations and ruin probabilities for solutions to stochastic recurrence equations with heavy-tailed innovations",
  journal		= "Ann. Probab.",
  volume		= "33",
  pages		= "1992--2035",
  year			= "2005",
  doi			= "10.1214/009117905000000350"
}

@article{Mikosch2000,
  title		={The supremum of a negative drift random walk with dependent heavy-tailed steps},
  author	={Mikosch, T. and Samorodnitsky, G.},
  journal	={Ann. Appl. Probab.},
  year		={2000},
  volume	={10},
  pages	={1025--1064},
  doi		= "10.1214/AOAP/1019487517"
}

\end{document}